\documentclass[11pt]{amsart}
\usepackage{amsmath, amsfonts, amssymb,amsthm}
\usepackage{euscript}
\usepackage[T1]{fontenc}
\usepackage{graphicx}
\usepackage{color}

\usepackage{hyperref}

\newtheorem{thmx}{Theorem}

\newtheorem{thm}{Theorem}[section]
\newtheorem{lem}[thm]{Lemma}

\newtheorem{prop}[thm]{Proposition}
\newtheorem{cor}[thm]{Corollary}

\theoremstyle{definition}
\newtheorem{defn}[thm]{Definition}

\newtheorem{ex}[thm]{Example}

\DeclareMathOperator{\R}{\mathbb R}

\DeclareMathOperator{\C}{\mathbb C}
\DeclareMathOperator{\N}{\mathbb N}

\DeclareMathOperator{\Zi}{\mathcal Z}
\DeclareMathOperator{\V}{\mathcal V}
\DeclareMathOperator{\I}{\mathcal I}

\DeclareMathOperator{\K}{\mathcal K}

\DeclareMathOperator{\QQ}{\mathbb Q}

\DeclareMathOperator{\supp}{supp}
\DeclareMathOperator{\Sing}{Sing}

\DeclareMathOperator{\Cent}{Cent}

\DeclareMathOperator{\RCent}{C-Spec}
\DeclareMathOperator{\RSp}{R-Spec}

\DeclareMathOperator{\sta}{s}
\DeclareMathOperator{\Reg}{Reg}

\DeclareMathOperator{\Sp}{Spec}

\def \S {{\mathcal S}}

\DeclareMathOperator{\p}{\mathfrak{p}}
\DeclareMathOperator{\q}{{\mathfrak q}}

\DeclareMathOperator{\Max}{\rm Max}
\DeclareMathOperator{\ReMax}{\rm R-Max}
\DeclareMathOperator{\CentMax}{\rm C-Max}

\DeclareMathOperator{\m}{\mathfrak m}

\DeclareMathOperator{\ID}{Ideal}
\DeclareMathOperator{\CO}{Cone}

\def\RadCe {\sqrt[C]}
\def \S {{\mathcal S}}
\def \C {{\mathcal C}}

\def \RR {{\mathbb R}}

\begin{document}

\title[\tiny{On central orderings}]{On central orderings}
\author{Goulwen Fichou, Jean-Philippe Monnier and Ronan Quarez}
\thanks{The authors have received support from the Henri Lebesgue Center ANR-11-LABX-0020-01 and the project EnumGeom ANR-18-CE40-0009.}

\address{Goulwen Fichou\\
Univ Rennes, CNRS, IRMAR - UMR 6625, F-35000 Rennes, France}
\email{goulwen.fichou@univ-rennes.fr}

\address{Jean-Philippe Monnier\\
   LUNAM Universit\'e, LAREMA, Universit\'e d'Angers}
\email{jean-philippe.monnier@univ-angers.fr}

\address{Ronan Quarez\\
Univ Rennes\\
Campus de Beaulieu, 35042 Rennes Cedex, France}
\email{ronan.quarez@univ-rennes.fr}
\date\today

\begin{abstract} 
	We define the notion of central orderings for a general commutative ring $A$ which generalizes the notion of central points of irreducible real algebraic varieties. We study a central and a precentral loci which both live in the real spectrum of the ring $A$ and allow to state central Positivestellens\"atze in the spirit of Hilbert 17th problem.
\end{abstract}

\maketitle

\section{Introduction}

Let $R$ be a real closed field.
The famous Hilbert 17th problem, answered by Artin in 1927, states that a polynomial $f\in R[x_1,\ldots,x_n]$ is nonnegative on $R^n$ if and only if it 
is a sum of squares of rational functions, namely $f=\sum_{i=1}^r f_i^2$ where $f_i\in R(x_1,\ldots,x_n)$. We denote by $\sum \K(V)^2$ the sum of squares of rational functions.

Real algebra was developed to solve Hilbert 17th problem and also to formulate general
positivstellens\"atze for polynomials nonnegative on a given closed semialgebraic subset $W=\{f_1\geq 0,\ldots f_r\geq 0\}$ of $R^n$. Among theses positivstellens\"atze, one notably recovers a real nullstellensatz. 

It has been possible to generalize these positivstellens\"atze for functions in the coordinate ring $R[V]$ of an irreducible affine algebraic variety $V$ over $R$ (see \cite[Cor. 4.4.3]{BCR}). 
But, as noted in \cite[Example 6.1.8]{BCR}),
one difference with the original Hilbert 17th problem is that a polynomial function $f\in R[V]$ which can be written as a sum of squares in the fraction field $\K (V)$ may not be nonnegative on the whole set of real closed points $V(R)$ of $V$. It appears that such a positivstellensatz certifies positivity only on $\Cent(V(R))$, the central locus of $V$, which consists in the Euclidean closure of the nonsingular real closed points. More precisely, it provides an equivalence in the statement of Hilbert 17th property, in the sense that given $f\in\R[V]$:
$$f(\Cent(V(R)))\geq 0\iff f\in\sum \K(V)^2.$$

The origin of the adjective "central", which can be considered as the key notion in this paper, comes from \cite{Du}. 
This notion appears several times in \cite{BCR} where the link with the positivity of sums of squares of rational functions is noted. A theory of seminormalization of real algebraic varieties adapted to the central locus is developed in \cite{FMQ-futur2} and continued in \cite{Mnew} where the definition of central ideal is introduced.

The principal goal of this paper is to lay the fondations of central algebra. Let $A$ be a commutative domain with fraction field denoted by $\K(A)$. We study two subspaces of the set of all orderings in $A$, also called the real spectrum $\Sp_r A$ of $A$. These two subspaces are
adapted to the central locus of real varieties in order to get some
central positivstellens\"atze, a central nullstellensatz and a central Hilbert 17th property.

The paper is organized as follows. We start in section 2 with a reminder on real algebra and in particular the real spectrum of a ring as introduced in \cite{BCR}. We are particularly interested in the support mapping from the set of cones of $A$ to its set of ideals convex with respect to the cone of sums of squares.

In section 3, we study central ideals which have already been defined in \cite{Mnew}. This subcategory of real prime ideals has recently been used in \cite{FMQ-futur2} to develop the theory of central seminormalization (real version of the seminormalization introduced by Traverso \cite{T}). The motivation was the property that central ideals (that are in particular real ideals) behave much better than real ideals when we consider integral extensions of rings. Similarly to Dubois notion of central point of a real algebraic variety, we consider the notion of central orderings introduced in \cite{BP}, as the elements of $\Sp_r A$ which are in the closure of $\Sp_r \K(A)$ for the topology of the real spectrum. The set of central orderings, denoted by $\Sp_c A$, is a closed subset of $\Sp_r A$ and the supports of central orderings are exactly the central prime ideals of $A$.
We study these central orderings whose definition of topological nature is not easy to handle in order to prove algebraic statements as positivstellens\"atze. This motivates us to introduce another sort of orderings which we call precentral and are defined by a simple and natural algebraic condition. The precentral orderings are those orderings which contain the cone $A\cap \sum \K(A)^2$ and hence are a sup-class of central orderings. The set of precentral orderings, denoted by $\Sp_{pc }A$, is also a closed subset of $\Sp_r A$ and the supports of precentral orderings are again exactly the central prime ideals of $A$.

In section 4, we study the differences between central and precentral orderings, giving characterizations of these two kinds of orderings. Although theses orderings are distinct in general, it appears that they coincide for real algebraic varieties of dimension less than or equal to two.

In section 5, we give some precentral Positivstellens\"atze which comes naturally from the algebraic nature of precentral orderings. One of the main results of the paper is as follows.

\begin{thmx}\label{thmA}
Let $f_1,\ldots,f_r$ in $A$ and $f\in A$.
	Denote by $P\subset A$ the cone $(A\cap \sum \K(A)^2)[f_1,\ldots,f_r]$ and by 
	$\Lambda \subset \Sp_{pc} A$ the set $\{\alpha\in\Sp_{pc} A\mid f_1(\alpha)\geq 0,\ldots, f_r(\alpha)\geq 0\}$. Then
		\begin{enumerate}
\item $f\geq 0$ on $\Lambda$ if and only if $fq=p+f^{2m}$ for some $p,q$ in $P$ and $m\in \N$.
\item $f>0$ on $\Lambda$ if and only if $fq=1+p$ for some $p,q$ in $P$.		
\item $f=0$ on $\Lambda$ if and only if $f^{2m}+p=0$
for some $p$ in $P$ and $m\in \N$.		
	\end{enumerate}	
\end{thmx}

As a consequence, we obtain also some central positivstellens\"atze when the positivity conditions on central and precentral orderings coincide.
In particular, we get geometric central positivstellens\"atze for algebraic varieties of dimension less than or equal to two.

The study done in section 4 shows that we cannot differentiate central and precentral orderings by the global positivity of a single function. It enables to state a general version of Hilbert 17th property.
\begin{thmx}\label{thmB}
Let $f\in A$. The following properties are equivalent :
\begin{enumerate}
\item $f\geq 0$ on $\Sp_c A$.
\item $f\geq 0$ on $\Sp_{pc} A$.
\item $f\in\sum \K(A)^2$.
\end{enumerate}\end{thmx}

Note that when $A$ is the coordinate ring of an irreducible affine algebraic variety $V$ over a real closed field $R$, the previous properties are equivalent to  $f\geq 0$ on $\Cent(V(R))$.

Using the abstract formalism developed here, we are able to extend our positivstellens\"atze to other geometric settings than real algebraic varieties, namely the Nash and the real analytic settings.

The final section 6 deals with continuous rational functions. As previously recalled, one knows that a nonnegative $f\in R[x_1,\ldots,x_n]$ on $R^n$ is a sum of squares of rational functions. From \cite{Kre} it appears that $f$ is in fact a sum of squares of rational functions which can be extended continuously to the whole $R^n$. 
We study then the question of adding a continuity property in the third property of Theorem \ref{thmB}, and prove surprisingly that it is not always possible. Anyway, we establish a continuous central Hilbert 17th property when the non-negativity is assumed on the whole real spectrum.

\vskip 5mm

In all the paper, $R$ denotes a real closed field and all the rings are commutative and contain $\QQ$.

\section{Preliminaries on real algebra}
In this section we revisit the real algebra (introduced in \cite{BCR} and \cite{L}) from the angle of ideals convex with respect to the cone of sums of squares.

In this section $A$ is a ring.

\subsection{Preordering, convexity, convex and real ideals and the support mapping}

\begin{defn}
A cone of $A$ is a subset $P$ of $A$ such that $P+P\subset P$, $P\cdot P\subset  P$ and $A^2\subset P$. A cone $P$ is called proper if $-1\not\in P$.
\end{defn}

Note that the set $\sum A^2$ of sums of squares is the smallest cone of $A$. In case $-1\not\in\sum A^2$, we say that $A$ is a formally real ring, which means also that it admits a proper cone.
Another example of major interest in the paper is, if $A$ is an integral domain with fraction field $\K(A)$, the cone $A\cap \sum \K(A)^2$ of elements in $A$ that are sum of squares of elements in $\K(A)$. This cone plays a crucial role in the paper, it will be denoted simply by $\C=A\cap \sum \K(A)^2$.

We will encounter the notion of cone generated by a subset. Let $P$ be cone of $A$. If $S\subset A$ then $P[S]=\{\sum\limits_{i=1}^n t_is_i\mid t_i\in P,\,\,s_i\in S\}$ is the smallest cone of $A$ containing $P$ and $S$. If $S=\{f_1,\ldots,f_k\}$ then we also denote $P[S]$ by $P[f_1,\ldots,f_k]$.\\

Recall that for a given cone $P$ of $A$, the set $P\cap -P$ is called the support of $P$ and is denoted by $\supp(P)$.

\begin{prop} 
\label{commute1} 
We have a support map $$\supp: \CO (A)\to\ID(A), \,\,P\mapsto \supp(P)$$ 
which preserves inclusions.

Let $A\to B$ be a ring morphism. The diagram $$\begin{array}{ccc}
	\CO(B)&\stackrel{\supp}{\rightarrow}  & \ID(B) \\
	\downarrow&&\downarrow \\
	\CO(A)& \stackrel{\supp}{\rightarrow} & \ID(A)\\
\end{array}$$ is commutative, where the vertical arrows are the natural maps $\ID(B)\to \ID(A)$, $I\mapsto \varphi^{-1}(I)$, and $\CO(B)\to \CO (A)$, $P\mapsto \varphi^{-1}(P)$.
\end{prop}

\begin{proof}
The fact that the support map is well-defined follows directly from the formula
$$xy=\frac14 x(y+1)^2-\frac14 x(y-1)^2$$
for $x,y\in A$. The commutativity of the diagram is straightforward.
\end{proof}

Note that the support map sends a proper cone on a proper ideal. 
We are interested more generally by  characterizing the image of the support map. Note that this map is in general not surjective, for example the prime ideal $(x^2+1)\subset\RR[x]$ is not the support of a cone otherwise this one would not be proper.

We recall to this aim the notion of convexity of an ideal related to a given cone \cite{BCR}.
\begin{defn}
Let $P$ be a cone of $A$.
An ideal $I$ of $A$ is called $P$-convex if 
$$p_1+p_2\in I\,\,{\rm with}\,\,p_1\in P\,\,{\rm and}\,\,p_2\in P\Rightarrow p_1\in I\,\,{\rm and}\,\,p_2\in I.$$
\end{defn}

The support of a cone $P$ is always convex for this cone, and it is easy to check that it is even the smallest $P$-convex ideal.


We give an elementary property about convexity that will be useful in the sequel.

\begin{lem} \label{spec+conv0}
Let $P$ and $Q$ be cones of $A$ with $P\subset Q$, and $I\subset A$ be a $Q$-convex ideal. Then $I$ is $P$-convex.
\end{lem}

The following result is useful to study the image of the support map.

\begin{lem} \label{spec+conv1}
Let $P$ be a cone and $I$ be an ideal of $A$. There exists a cone $Q$ of $A$ such that $P\subset Q$ and $\supp(Q)=I$ if and only if $I$ is $P$-convex. In this situation, $I+P$ is the smallest cone containing $P$ with support $I$ and it satisfies $I+P=P[I]$.
\end{lem}

\begin{proof} Assume that $I$ is $P$-convex. The point is to prove that $\supp I+P=I$. To prove the non-obvious inclusion, let $q=a+b\in\supp(I+P)$ with $a\in I$ and $b\in P$. So $q=a+b\in -(I+P)$ and thus $b\in -(I+P)$. We have $b=-a'-b'$ with $a'\in I$ and $b'\in P$ and it follows that $b+b'\in I$. Since $I$ is $P$-convex then $b\in I$ and thus $q\in I$. It proves $I$ is the support of $I+P$.

The converse implication comes from Lemma \ref{spec+conv0}.
\end{proof}


We answer now to the question asked above concerning the image of the support map.
\begin{thm} \label{suppmap1}
The image of the support map $\supp: \CO (A)\to\ID(A)$ is the set of $\sum A^2$-convex ideals of $A$.
\end{thm}

\begin{proof} 
For $P$ a cone of $A$, $\supp(P)$ is $P$-convex and thus $\sum A^2$-convex since $\sum A^2\subset P$.

Let $I$ be a $\sum A^2$-convex ideal of $A$. Then the cone $I+\sum A^2$ is a cone with support $I$ from Lemma \ref{spec+conv1}.
\end{proof}



There exists a notion of radical ideal with respect to a cone which is no more than the convexity with respect to the cone plus the classical radicality.
\begin{defn}
Let $P$ be a cone of $A$.
An ideal $I$ of $A$ is called $P$-radical if 
$$a^2+p\in I\,\,{\rm with}\,\,a\in A\,\,{\rm and}\,\,p\in P\Rightarrow a\in I$$
\end{defn}

It means equivalently that the ideal is radical and $P$-convex by \cite[Prop. 4.2.5]{BCR}. For instance a real ideal, which is by definition a $\sum A^2$-radical ideal, is radical and $\sum A^2$-convex. 
Our interest in the notion of $\sum A^2$-convex ideals is motivated by the natural feeling that some non real ideals seem closer to be real (like the ideal $(x^2)$ in $\R[x]$) than others (e.g the ideal $(x^2+1)$ in $\R[x]$). And indeed one may check that the ideal $(x^2)$ is $\sum \R[x]^2$-convex.

\subsection{Orderings, real and Zariski spectra and the support mapping}

We denote by $\Sp A$ (resp. $\RSp A$)
the (resp. real) Zariski spectrum of $A$, i.e the set of all
(resp. real) prime ideals of $A$. The set of maximal (resp. and real)
ideals is denoted by $\Max A$ (resp. $\ReMax A$). We endow $\Sp A$
with the Zariski topology (whose closed sets are) generated by the sets $\V(f)=\{\p\in\Sp A\mid 
f\in\p\}$ for $f\in A$. The
subsets $\RSp A$, $\Max A$ and $\ReMax A$ of $\Sp A$ are endowed with
the induced Zariski topology.
We denote also $\V(I)=\{\p\in\Sp A\mid 
I\subset\p\}$ for $I$ an ideal of $A$.

For $\p\in\Sp A$, we denote by $k(\p)$ the residue field at $\p$ i.e the fraction field of $A/\p$.

\begin{defn}
A proper cone $P$ is called an ordering if it satisfies 
$$ab\in P\Rightarrow a\in P\,\,{\rm or}\,\,-b\in P.$$
The set of orderings of $A$ is denoted by $\Sp_r A$.
\end{defn}

We recall the principal properties of orderings.
\begin{prop} \cite{BCR}
\label{primeordering}
Let $P$ be an ordering of $A$. We have
\begin{enumerate}
\item $P\cup-P=A$.
\item $\supp(P)$ is a real prime ideal of $A$.
\item $\overline{P}=\{ \overline{a}/\overline{b} \in k(\supp(P))\mid ab\in P\}$ is an ordering of $k(\supp(P))$ such that $P=\varphi^{-1} (\overline{P})$ with $\varphi :A\to k(\supp(P))$ the canonical morphism and $\overline{a}$, $\overline{b}$ denote the classes of $a$ and $b$ in $A/\supp(P)$.
\item There exists a morphism $\alpha:A\to R_\alpha$ such that $R_\alpha$ is a real closed field, $\ker\alpha=\supp(P)$ and $P=\phi^{-1}((R_\alpha)_+)$.
\end{enumerate}
Conversely, given $\alpha:A\to R_\alpha$ a morphism into a real closed field, $P_\alpha=\alpha^{-1}((R_\alpha)_+)$ is an ordering of $A$ with support $\ker\alpha=\supp(P_\alpha)$.
\end{prop}

Thus one can see an ordering of $A$ equivalently as a morphism into a real closed field. We will use this identification in all the paper.

By \cite[Thm. 4.3.7]{BCR}, $A$ is formally real if and only if $\Sp_r A\not=\emptyset$ if and only if $\RSp A\not=\emptyset$.
Let $\alpha\in\Sp_r A$. Let $a\in A$, we set $a(\alpha)\geq 0$ if $a\in P_\alpha$, $a(\alpha)> 0$ if $a\in P_\alpha\setminus\supp(P_\alpha)$, $a(\alpha)= 0$ if $a\in \supp(P_\alpha)$.
A set of the form $$\S(f_1,\ldots,f_k)=\{ \alpha\in\Sp_r A\mid f_1(\alpha)>0,\ldots, f_k(\alpha)>0\}$$
for $f_1,\ldots,f_k$ some elements of $A$,  is called a basic open subset of $\Sp_r A$. A basic open subset of the form $\S(f)$, for a $f\in A$, is called principal.
The real spectrum $\Sp_r A$ is a topological space for the topology (whose open sets are) generated by the basics open sets. In the sequel, if $T$ is a subset of $\Sp_r A$ then we will denote by $\overline{T}$ the closure of $T$ for the real spectrum topology. A constructible subset of $\Sp_r A$ is a finite boolean combination of basic open sets.

Given two orderings $\alpha$ and $\beta$, one says that $\beta$ specializes to $\alpha$, and write $\beta\to\alpha$, when $P_\alpha \subset P_\beta$. An equivalent characterization from \cite[Prop. and Defn. 7.1.18]{BCR} is $(P_\alpha\setminus(-P_\alpha))\subset (P_\beta\setminus(-P_\beta))$, and another is $\alpha\in\overline{\{\beta\}}$.

As a consequence closed subsets of the real spectrum are closed by specialization, and the converse is also true for constructible closed subsets \cite[Prop. 7.1.21]{BCR}.

The restriction of the support map $\supp: \CO (A)\to\ID(A)$ to orderings gives a map
$\supp:\Sp_r A\to \Sp A$ whose image is contained in $\RSp A$. We complete here the study of this support map initiated in Theorem \ref{suppmap1} and Proposition \ref{commute1}.
\begin{prop} \label{suppmap2}
The support map $\supp:\Sp_r A\to \Sp A$ is continuous and its image is $\RSp A$.

A morphism $\varphi: A\to B$ induces natural maps $\RSp(B)\to \RSp(A)$ and $\Sp_r B\to \Sp_r A$ and hence a commutative diagram:
$$\begin{array}{ccc}
	\Sp_r B&\stackrel{\supp}{\rightarrow}  & \RSp B \\
	\downarrow&&\downarrow \\
	\Sp_r A&\stackrel{\supp}{\rightarrow} & \RSp A\\
\end{array}$$ 
\end{prop}

\begin{proof}
Let $\p$ be a real prime ideal. Since $\p$ is $\sum A^2$-convex then it follows from   \cite[Prop. 4.3.8]{BCR} that there exists an ordering with support equal to $\p$. It proves $\supp(\Sp_r A)=\RSp A$. 

The continuity follows from \cite[Prop. 7.1.8]{BCR} or \cite[Prop. 4.11]{L}.
\end{proof}

For the convenience of the reader, we recall from \cite[Prop. 4.4.1]{BCR} the formal Positivstellensatz, a key tool that we will use several times in the paper.

\begin{thm}\label{FormalPSS}
	Let $A$ a commutative ring. In $A$ consider a subset $H$, a monoid $M$ generated by the $(b_j)_{j\in L}$ and an ideal $I$ generated by the $(c_k)_{k\in T}$.

There is no $\alpha\in \Sp_r A$ such that 
$H\subset P_\alpha$, $\forall j\in L$ $b_j\notin \supp( \alpha)$, and $\forall k\in T$ $c_k\in \supp (\alpha)$ if and only if we have an identity
$$p+b^2+c=0$$
where $p\in \sum A^2[H]$, $b\in M$, $c\in I$.
\end{thm}

\subsection{Real spectrum of geometric rings}

Let us recall how the real points of a variety are related to the real spectrum of its coordinate ring. Assume $V=\Sp R[V]$ is an affine algebraic variety over $R$ with coordinate ring $R[V]$. We denote by $V(R)$ the set of real closed points of $V$ i.e the subset of $\p\in V$ such that $k(\p)=R$. We have inclusions $$V(R)\hookrightarrow \RSp R[V]\hookrightarrow\Sp R[V]$$ that makes $V(R)$ a topological space for the (induced) Zariski topology. The real zero sets $\Zi(f)=\V(f)\cap V(R)$, for $f\in R[V]$, generate the closed subsets of $V(R)$ for the Zariski topology. If $T$ is a subset of $V(R)$ then we will denote by $\overline{T}^Z$ the closure of $T$ for the Zariski topology. Since $R[V]=R[x_1,\ldots,x_n]/I$ for a radical ideal $I\subset R[x_1,\ldots,x_n]$ then we get an inclusion $$V(R)\hookrightarrow R^n$$ that identifies $V(R)$ as a closed subset of $R^n$ for the Zariski topology and also for the Euclidean topology. 
Recall that the unique ordering on $R$ gives rise to an order topology on the affine spaces $R^n$ called the Euclidean topology \cite{BCR}, in a similar way than the Euclidean topology on $\R^n$, even if the topological space $R$ is not connected (except in the case $R=\R$) or the closed interval $[0,1]$ is in general not compact. If $T$ is a subset of $V(R)$ then we will denote by $\overline{T}^E$ the closure of $T$ for the Euclidean topology.  An element of $V(R)$ can also be seen as a morphism from $R[V]$ to $R$. So we get a third inclusion $$V(R)\hookrightarrow \Sp_r R[V]$$
that identifies $V(R)$ as a closed (by specialization) subset of $\Sp_r R[V]$. For $x\in V(R)$, we denote by $\alpha_x:R[V]\to R$, $f\mapsto f(x)$ the associated ordering of $R[V]$.
A set of the form $$S(f_1,\ldots,f_k)=\{ x\in V(R)\mid f_1(x)>0,\ldots, f_k(x)>0\}$$
for $f_1,\ldots,f_k$ some elements of $R[V]$,  is called a basic open subset of $V(R)$, it is an open subset of $V(R)$ for the Euclidean topology. 
A basic open subset of the form $S(f)$, for a $f\in R[V]$, is called principal. Clearly, the principal open subsets generates the Euclidean topology. A semialgebraic subset of $V(R)$ is a finite boolean combination of basic open sets. By \cite[Prop. 7.2.2 and Thm. 7.2.3]{BCR}, the inclusion $V(R)\hookrightarrow \Sp_r R[V]$ induces a one-to-one map between the semialgebraic subsets of $V(R)$ and the constructible subsets of $\Sp_r R[V]$, this map sends the semialgebraic set $S$ to the constructible $\widetilde{S}$ described by the same inequalities than $S$. For $f_1,\ldots,f_k$ in $R[V]$ we have $\widetilde{S(f_1,\ldots,f_k)}=\S(f_1,\ldots,f_k)$. One important property of this map is the commutation with the closures for the Euclidean topology and the real spectrum topology \cite[Thm. 7.2.3]{BCR}, namely for a semialgebraic subset $S$ of $V(R)$ we have
$$\widetilde{(\overline{S}^E)}=\overline{(\widetilde{S})}.$$

\subsection{Stability index}\label{StabIndex}
The material of this subsection will be used in section \ref{CentvsPreccent}, hence the reader may momentarily skip it until reaching this section.

\begin{defn}
The stability index of $A$ is the infimum of the numbers $k\in\N$ such that for any basic open subset $S$ of $\Sp_r A$ there exist $f_1,\ldots,f_k$ in $A$ such that $S=\S(f_1,\ldots,f_k)$.

Similarly, the stability index $\sta(U)$ of an open subset $U$ of $\Sp_r A$ is the infimum of the numbers $k\in\N$ such that for any basic open subset $S$ of $\Sp_r A$ such that $S\subset U$ there exist $f_1,\ldots,f_k$ in $A$ such that $S=\S(f_1,\ldots,f_k)$. 
 If $U=\emptyset$ then we set $\sta(U)=0$.
\end{defn}

When $V=\Sp R[V]$ is an affine algebraic variety over $R$ with coordinate ring $R[V]$, from the properties of the map $S\mapsto \widetilde{S}$ exposed previously then, it is clear that the stability index of $R[V]$ is also the infimum of the numbers $k\in\N$ such that for any basic open subset $S$ of $V(R)$ there exist $f_1,\ldots,f_k$ in $R[V]$ such that $S=S(f_1,\ldots,f_k)$. In that case, the stability index of $R[V]$ is also called the stability index of $V(R)$.

We recall the famous theorem of Bröcker \cite{Br} and Scheiderer \cite{She}: 
\begin{thm} (Bröcker-Scheiderer)\\
Let $V=\Sp R[V]$ be an affine algebraic variety over $R$ with coordinate ring $R[V]$. Then, the stability index of $R[V]$ coincides with 
the stability index of $V(R)$ and is equal to the dimension of $V(R)$ (as a semialgebraic set).
\end{thm}

Noe that in the case of finitely generated algebras over a non real closed field, the formula is not as simple. Concerning the stability index of abstract rings, we refer to \cite{ABR}.


\section{Central algebra}
From now on, the ring $A$ is assumed to be a domain with fraction field $\K(A)$. Classical real algebra is developed around the structural cone $\sum A^2$. It is a fruitful tool to make a link between algebra and the geometry of the real points of a variety. In this section, we develop a notion of central algebra in order to take into account the central points of a real variety, i.e. the Euclidean closure of the nonsingular real closed points. The central locus of a real algebraic variety has been defined in \cite{BCR} inspired by the work of Dubois \cite{Du}. This central algebra is built around the cone $\C=A\cap \sum \K(A)^2$ of the sums of squares of elements of the fraction field that belong to the ring $A$.

Note that the word {\it central} already appeared in the literature in an algebraic context : a notion of central ideal is introduced in \cite{Mnew}, a definition of central ordering is given in \cite{BP}. Our goal is to show that the central algebra unifies these notions, and is a good framework to state abstract central Positivstellens\"atze.

\subsection{Cones and orderings with support the null ideal}

In this section we are interested by describing the inverse image of the null ideal by the support maps $\CO (A)\to \ID(A)$ and $\Sp_r A\to \Sp A$. 

Remark that the cones (resp. orderings) of $\K(A)$ with support the null ideal are exactly the proper cones (resp. the orderings) of $\K(A)$. Now we aim to compare the proper cone of $\K(A)$ with the cone of $A$ with support the null ideal.

\begin{prop} \label{orderingnull}
The map $P\mapsto P\cap A$ sends injectively the proper cones (resp. the orderings) of $\K(A)$ into the cones (resp. the orderings) of the ring $A$ with support the null ideal.

The map between $\Sp_r \K(A)$ and the set of orderings of $A$ with support $(0)$ given by $P\mapsto P\cap A$ is bijective and the inverse map
	is given by $$Q\mapsto Q_{\K(A)}:=\{a/b\in\K(A)\mid ab\in Q\}.$$
\end{prop}

\begin{proof}
The first point comes from the property of the support map, cf. Proposition \ref{commute1} and Proposition \ref{suppmap2}. The second point is a consequence of (3) of Proposition \ref{primeordering}.
\end{proof}

In the sequel, we identify the proper cones of $\K(A)$ with a subset of $\CO( A)$ with support the null ideal and $\Sp_r \K(A)$ with the subset of $\Sp_r A$ with support the null ideal.
We illustrate in the following example the non-surjectivity in the case of cones.

\begin{ex} \label{cubic} Let $A=\R[V]$ be the coordinate ring of the cubic curve $V$ with an isolated point, namely $A=\R[x,y]/(y^2-x^2(x-1))$. 
It is easy to see that the cone $\sum A^2$ has support the null ideal, however it does not come from a cone of $\K(A)$. Indeed, assume that $\sum A^2=P\cap A$ for a cone $P$ of $\K(A)$. We have $x-1=(y/x)^2\in \K(A)^2$ and thus $x-1\in P\cap A=\sum A^2$. It follows that $x-1$ must be nonnegative on $V(\R)$ and evaluating at the isolated point we get a contradiction.
\end{ex}

\subsection{$\C$-convex and central ideals}

We also recall the definition of central ideal introduced in \cite{Mnew}.

\begin{defn} 
Let $I$ be an ideal of $A$. Then 
$I$ is called central if $I$ is $\C$-radical. The
  central radical of $I$ is defined as
  $$\RadCe I=\{a\in A|\,\, \exists m\in\N \,\, \exists b\in \C\,\, {\rm such\,\,that}\,\,a^{2m}+b\in I
  \}.$$
We denote by $\RCent A$ (resp. $\CentMax A$) the subset of $\Sp A$ of central prime (resp. maximal) ideals.
\end{defn}

From \cite[Prop. 4.2.5]{BCR}, we know that an ideal is central if and only if it is radical and $\C$-convex.
A $\C$-convex (resp. central) ideal is $\sum A^2$-convex (resp. real) but the converse is not true as illustrated by the ideal $I=(x,y)$ in Example \ref{cubic}. Indeed we have $1+(y/x)^2=x\in I$, $1\in \C$ and $(y/x)^2\in\C$ but $1\not\in I$. So $I$ is a real ideal which is not $\C$-convex.

 By \cite[Prop. 3.14]{Mnew}, $\RadCe I$ is the intersection of the central prime ideals containing $I$ and moreover $I$ is central if and only if $I=\RadCe I$.

We give several characterizations of the existence of a central ideal in a domain.
 \begin{prop} 
\label{existcentralideal}
The following properties are equivalent:
\begin{enumerate}
\item $\K(A)$ is a formally real field.
\item $\RCent A\not=\emptyset$.
\item $A$ has a proper $\C$-convex ideal.
\item $A$ has a proper central ideal.
\item $(0)$ is a central ideal of $A$.
\end{enumerate}
\end{prop}

\begin{proof}
The equivalence between (1), (2), (4), (5) follows from \cite[Prop. 3.16]{Mnew}. Clearly (4) implies (3). Let $I$ be a proper $\C$-convex ideal. We claim that $\RadCe I$ is also proper and the proof will be done. Assume $1\in \RadCe I$. There exists $b\in \C$ such that $1+b\in I$ and since $I$ is $\C$-convex then $1\in I$, a contradiction.
\end{proof}

In the case where $A$ is the coordinate ring of an irreducible affine algebraic variety $V$ over $R$, the existence of a central ideal is equivalent to the existence of a so-called central point. From \cite[Defn. 7.6.3]{BCR}, the central locus of $V(R)$ or $V$ and denoted by 
$\Cent V(R)$, is defined to be the closure for the Euclidean topology of the nonsingular real closed points i.e $\Cent V(R)=\overline{V_{reg}(R)}^E$. In the sequel, we say that $V$ is a central variety if $\Cent V(R)=V(R)$. It follows from the definition that a nonsingular variety is central. On the contrary, the isolated point in the cubic exhibited in Example \ref{cubic} is non-central. Note that $\Cent V(R)$ is a closed semialgebraic set since $V_{reg} (R)$ and the Euclidean closure of a semialgebraic set remains semialgebraic \cite[Prop. 2.2.2]{BCR}.

The definition of central ideals gives a new formulation of the Central Nullstellensatz stated in \cite[Cor. 7.6.6]{BCR}. 

\vskip 0,2cm\noindent\textbf{Theorem}(Central Nullstellensatz)\\ 
  Let $V$ be an irreducible affine algebraic variety over
  $R$. Then:
  $$I\subset R[V]\,\,{\rm is\,\, a\,\, central\,\, ideal}\,\, \Leftrightarrow\,\, I=\I(\Zi(I)\cap
  \Cent V(R))\,\, \Leftrightarrow\,\, I=\I(\V(I)\cap
  \Cent V(R))$$
  In particular, we have $\CentMax R[V]=\Cent V(R)$.
\vskip 0,2cm

It furnishes a tool to decide geometrically whether an ideal is central. 

\begin{ex}\label{WhitneyEx}
Let $V$ be the Whitney umbrella given by the equation $y^2=zx^2$. Then $\p=(x,y)\subset \R[V]$ is a
central prime ideal since the stick $\Zi(\p)$ of
the umbrella meets $\Cent V(\R)$ in maximal dimension.
\end{ex}

\begin{ex}\label{CartanEx} Let $V$ be the Cartan umbrella given by the equation $x^3=z(x^2+y^2)$. Then $\p=(x,y)\subset \R[V]$
is a real prime ideal but not a
central ideal by the Central Nullstellensatz since the stick $\Zi(\p)$ of
the umbrella meets $\Cent V(\R)$ in a single point. Alternatively, one can show algebraically that $\p$ is not central by the identity 
$$b=x^2+y^2-z^2=x^2+y^2-\frac{x^6}{(x^2+y^2)^2}=\frac{3x^4y^2+3x^2y^4+y^6}{(x^2+y^2)^2}\in \R[V] \cap \sum\K(V)^2.$$
Indeed $z^2+b=x^2+y^2\in\p$ but $z\not\in\p$.
\end{ex}

One goal in the paper is to generalize this Central Nullstellensatz to Central Positivstellens\"atze in order to get algebraic certificates of positivity on subsets of the central locus.

\subsection{Central cones and precentral orderings}
In this section we introduce the notion of central cones whose supports give the $\C$-convex ideals.

\begin{defn}
\label{defprecent}
A cone $P\subset A$ is called central if there exists a cone $Q$ of $\K(A)$ such that $(Q\cap A)\subset P$. We denote by $\CO_{c} (A)$ the subset of all central cones in $\CO (A)$. 

An ordering which is a central cone is called a precentral ordering.
We denote by $\Sp_{pc} A$ the subset of all precentral orderings in $\Sp_r A$. We say that $A$ is precentral if $\Sp_{pc} A=\Sp_r A$. 
\end{defn}

Since any cone of $\K(A)$ contains $\sum\K(A)^2$, it follows from the definition that a cone $P$ of $A$ is central if and only if $\C\subset P$. 
This allows one to write 
$$\Sp_{pc}A=\bigcap_{f\in \C}\{\alpha\in\Sp_r A\mid f(\alpha)\geq 0\}.$$
This shows that $\Sp_{pc}A$ is a closed subset in $\Sp_rA$ as an intersection of closed subsets. Beware that it is not necessarily a constructible set as it will be point in the sequel.

We now study the restriction of our support map to the set of central cones and show that it coincides with the set of all $\C$-convex ideals.
\begin{prop} \label{support3}
Let $\supp: \CO(A)\to\ID(A)$ be the support map. We have 
\begin{enumerate}
\item $\supp(\CO_{c} (A))$ is the set of $\C$-convex ideals.
\item $\supp(\Sp_{pc} A)=\RCent A$.
\end{enumerate}
\end{prop}

\begin{proof}
Let $P$ be a central cone. Since $\C\subset P$ then $\supp(P)$ is $\C$-convex by Lemma \ref{spec+conv0}. Let $I$ be a $\C$-convex ideal.
Since $I$ is $\C$-convex then from Lemma \ref{spec+conv1} then $I+\C$ is a cone with support equal to $I$. Since $I+\C$ is clearly a central cone then we have proved (1).

From (1) then $\supp(\Sp_{pc} A)\subset\RCent A$. To show the converse implication, assume $\p$ is a central prime ideal. Then $\p$ is $\C$-convex and we conclude by using \cite[Prop. 4.3.8]{BCR}.
\end{proof}

Looking at the example \ref{WhitneyEx} of the Whitney umbrella, it is easy to see that a cone and even an ordering with support a $\C$-convex ideal is not always central as a cone.

\begin{prop}
\label{prop8}
An ordering of $A$ with support the null ideal is precentral i.e $\Sp_r \K(A)\subset \Sp_{pc} A$.
\end{prop}

\begin{proof}
Let $P\subset A$ be an ordering that $\supp(P)=(0)$. By Proposition \ref{orderingnull} then there exists $Q\in\Sp_r \K(A)$ such that $P=Q\cap A$. Since $\sum \K(A)^2 \subset Q$ then $\C \subset P$.
\end{proof}

From the example \ref{cubic} of the cubic, we know that a cone with support the null ideal is not always central i.e the statement of Proposition \ref{prop8} cannot be relaxed to cones.

One may give equivalent conditions on the existence 
of a precentral ordering:

\begin{prop}
\label{thm11}
The following properties are equivalent:
\begin{enumerate}
\item $\K(A)$ is a formally real field.
\item There is a proper cone in $\CO_{c}(A)$.
\item There exists a proper $\C$-convex ideal in $A$.
\item $\Sp_{pc} A\not=\emptyset$.
\item $\RCent A\not=\emptyset$.
\end{enumerate}
\end{prop}
\begin{proof}
One may use Propositions \ref{existcentralideal}, \ref{support3} and \ref{prop8}.
\end{proof}

Let us end this section by considering the ring $A=\R[x,y]$ which is clearly precentral, namely 
$\Sp_{pc} A=\Sp_r A$. 
However, $\sum A^2\not =\C$ (for instance consider the Motzkin polynomial)
and hence 
$\CO_{c} (A)\not=\CO( A)$.

\subsection{Central orderings}

We begin with recalling the definition of central ordering from \cite{BP}, definition which has inspired the definition \ref{defprecent} of a central cone in the preceding section.

\begin{defn}
\label{defcent}
A cone $P\in\Sp_r A$ is called central if there exists an ordering $Q\in\Sp_r \K(A))$ such that $(Q\cap A)\to P$. We denote by $\Sp_{c} A$ the subset of $\Sp_r A$ of central orderings.
We say that $A$ is central if $\Sp_{c} A=\Sp_r A$. 
\end{defn}

The notion of central ordering is a priori different from that of precentral ordering introduced in the preceding section, and studying the difference is a crucial issue in this paper.

In the geometric setting, let us see that the central spectrum is compatible with the notion of central points previously recalled. Again, this comes from results in \cite{BCR}:
\begin{prop}
	\label{alg+geomcentrality}
	Let $V$ be an irreducible affine algebraic variety over $R$. 
	Then, $$\Sp_c R[V]=\widetilde{\Cent V(R)}.$$
	Moreover, $V$ is central if and only if $R[V]$ is central.
\end{prop}

\begin{proof}
To show the first statement, one knows from
	\cite[Prop. 7.6.2]{BCR} that if $x\in \Cent V(R)$, then 
	$x$ is the specialization of an ordering in $\Sp_r\K(V)$, in other word, 
	$\widetilde{\Cent V(R)}\subset 
	\overline{\Sp_r\K(V)}$.
	The converse inclusion comes from \cite[Prop. 7.6.4]{BCR} which we recall the argument since we will need it in an abstract setting after.

For dimensional reasons, $\Sp_r \K(V)\subset \widetilde{V_{reg}(R)}$
	and thus, using that the tilda map commutes with the closures for the Euclidean topology and the real spectrum topology, 
	we get $$\Sp_c R[V]=\overline{\Sp_r\K(V)}\subset\overline{\widetilde{V_{reg}(R)}}=\widetilde{\overline{V_{reg}(R)}^E}=\widetilde{\Cent V(R)}.$$
	It follows that $\Sp_c R[V]=\widetilde{\Cent V(R)}.$
	
	Now, let us deduce the second statement.
	Assuming that $R[V]$ is central, then $V(R)=\Sp_r R[V]\cap V(R)=\Sp_c R[V]\cap V(R)=\Cent V(R)$. Conversely, assuming that $V$ is central, namely $V(R)=\Cent V(R)$, one gets 
	$(\Sp_r R[V]\setminus\Sp_c R[V])=(\widetilde{V(R)}\setminus\widetilde{\Cent V(R)})=\widetilde{(V(R)\setminus\Cent V(R))}=\emptyset$ and hence $R[V]$ is a central domain.
\end{proof}

An alternative way of saying that an ordering $P$ is central is to say that there is $Q\in \Sp_r A$ such that $\supp(Q)=(0)$ and $Q\to P$. 
Of course, any central ordering is central as a cone and thus it is a precentral ordering. 
Moreover, $\Sp_c A=\overline{\Sp_r \K(A)}$ is naturally a closed subset of $\Sp_r A$.

Let us see now how to mimic the geometric argument motivating our definition in the abstract case. 
As usual, if $A$ is noetherian then set $\Reg A$ to be the set of all prime ideal $\p$ in $A$ such that $A_{\p}$ is a regular local ring. The complementary $\Sing A$ of $\Reg A$ in $\Sp A$ is a nonempty closed subset for the Zariski topology whenever the ring $A$ satisfied the so-called property $(J1)$ \cite[§32 B]{Ma}. Note that excellent rings satisfy this condition.

If $A$ is excellent then $\Reg A$ and $\Sing A$ are Zariski constructible subsets of $\Sp A$. One may derive the associated constructible $\widetilde{\Reg A}$ and 
$\widetilde{\Sing A}$ in $\Sp_r A$.
Namely $\widetilde{\Sing A}=\{\alpha\in \Sp_r A\mid \supp(\alpha)\in \Sing A\}$ and 
$\widetilde{\Reg A}=\Sp_r A\setminus \widetilde{\Sing A}$.

We give an abstract version of \cite[Prop. 7.6.4]{BCR}.
\begin{prop}\label{abstract764}
	Let $A$ be an excellent domain.
	Considering the closure in $\Sp_r A$, one has 
	$$\overline{\Sp_r \K(A)}= \overline{\widetilde{\Reg(A)}}$$
\end{prop}
\begin{proof}
	We start by showing that $\Sp_r \K(A)\subset \widetilde{\Reg(A)}$.
	
	Let us assume that $\alpha\in \Sp_r \K(A)\setminus\widetilde{\Reg(A)}$. Then, $\supp(\alpha)\in \Sing(A)$ : 
	$I\subset \supp(\alpha)$, where $\Sing A=\V(I)$. This is impossible if $\supp(\alpha)=0$.
	
	It remains to show that
	$$\widetilde{\Reg(A)}\subset \overline{\Sp_r \K(A)}$$
	
	For this, we use an analogous of $(i)\implies (iii)$ from \cite[Prop. 7.6.2]{BCR}.
	
	Let $\alpha$ be an ordering in $A$ whose support $\p$ is in $\Reg(A)$.
	Since $A$ has finite dimension and $A_{\p}$ has same fraction field as $A$, we deduce 
	the existence of $\beta\in \Sp_ r \K(A)$
	which specializes to $\alpha$, by using \cite[Lem. 3.4]{ABR} which says that, given any regular local ring $A$ of dimension $d$, of residue field $k$ and fraction field $K$, any ordering on $k$ admits $2^d$ generalizations in $\Sp_ r A$ which are orderings in $K$.
\end{proof}
With this framework in an abstract setting we recover the usual geometric properties. 

The end of the section will consist in studying the restriction of the support mapping to central orderings.
Let us start with the following:
\begin{lem}
\label{specialisation}
Let $\beta\in\Sp_c A$ with support the null ideal. Let $\q$ be a prime ideal of $A$ which is $P_\beta$-convex. Then $$P_\alpha=\q+ P_\beta$$
is a central ordering of $A$ with support $\q$ which is a specialization of $P_\beta.$
\end{lem}

\begin{proof}
From Lemma \ref{spec+conv1} then $P_\alpha$ is a cone with support $\q$ which is a specialization of $P_\beta$. It follows that $P_{\alpha}$ is proper.

Since $\q$ is $P_{\beta}$-convex then we may easily show that $P_\alpha=\q\cup P_\beta$.
If $ab\in P_\alpha$ then it follows from the fact that $\q$ is prime and $P_\beta$ is an ordering that $a\in P_\alpha$ or $-b\in P_\alpha$. The proof is done.
\end{proof}

Is is not possible to differentiate the supports of precentral orderings from those of central ones:
\begin{prop} \label{support4}
Let $\supp: \Sp_r A\to\Sp A$ be the support map. We have 
$$\supp(\Sp_{c} A)=\RCent A.$$
\end{prop}

\begin{proof}
From Proposition \ref{support3} 
one has $\supp(\Sp_{c} A)\subset \RCent A$.

Assume $\p\in\RCent A$. By \cite[Prop. 4.2.9]{BCR} there exists an ordering $P'\in\Sp_r \K(A)$ such that $\p$ is $(P'\cap A)$-convex. To end the proof use Proposition \ref{orderingnull} and Lemma \ref{specialisation} where $\beta$ is associated to $(P'\cap A)$ and $\q=\p$.
\end{proof}

This allows us to say that the existence of a central ordering, or in other words the fact that 
$\Sp_c A\not=\emptyset$, is equivalent to the conditions given
in Proposition \ref{thm11}.

\section{Central versus precentral}\label{CentvsPreccent}

This section is the heart of the paper. We aim to compare central and precentral orderings, orderings which have the same supports by Propositions \ref{support3} and \ref{support4}.
Since a precentral ordering of $A$ contains a proper cone of $\K(A)$ and a central ordering contains an ordering of $\K(A)$, it follows that a central ordering is precentral as previously noted. However the converse implication does not hold, and the goal of this section is to study the difference.  

We assume in the sequel that $\K(A)$ is a formally real field since otherwise we do not have any precentral nor central ordering.

First note that for closed points of varieties, both notions coincide.
\begin{prop}
\label{closedpoints}
Let $V$ be an irreducible affine algebraic variety over $R$ and let $x\in V(R)$. Then $\alpha_x$ is central if and only if $\alpha_x$ is precentral.
\end{prop}

\begin{proof}
Assume $\alpha_x$ is precentral.  By Proposition \ref{support3} then $\m_x=\supp(\alpha_x)\in \RCent R[V]$. Since $R[V]/\m_x=R$ is real closed, $\alpha_x$ is the unique ordering of $\Sp_r R[V]$ with support $\m_x$ then it follows from Proposition \ref{support4} that $\alpha_x$ is central.
\end{proof}

One may readily generalize to an abstract setting:
\begin{prop}
	\label{ringclosedpoints}
	Let $\alpha\in \Sp_r A$ be such that the residue field
	$k(\supp(\alpha))$ admits a unique ordering. Then, $\alpha$ is central if and only if $\alpha$ is precentral.
\end{prop}

Since $\Sp_{c} A\subset \Sp_{pc} A\subset\Sp_r A$, we already know that $A$ is precentral whenever $A$ is central and the latter condition is satisfied for instance when $A$ is the coordinate ring of an irreducible affine algebraic variety over $R$ which is central (see Proposition \ref{alg+geomcentrality}).

Our aim is to give characterizations of central and precentral orderings. To start, let us note that, since 
$\Sp_cA$ and $\Sp_{pc} A$ are closed subsets, a point $\alpha$ belongs to 
$\Sp_c A$ (resp. $\Sp_{pc} A$) if and only if $U\cap \Sp_c A\not=\emptyset$ (resp. $U\cap \Sp_{pc} A\not=\emptyset$) for any open subset $U$ containing $\alpha$. One may replace in that statement $U$ with basic open subsets like $\tilde{\S}(f_1,\ldots,f_k)$ which are a basis of neighbourhoods.

\begin{lem} \label{equivcentralbasis}
	Let $\alpha\in\Sp_rA$ and $f_1,\ldots,f_k$ in $A\setminus\{0\}$ such that $\alpha\in\S(f_1,\ldots,f_k)$.
	Let us consider the following  properties:
	\begin{enumerate}
		\item $A\cap \sum \K(A)^2[f_1,\ldots,f_k]\subset P_\alpha$.
		\item $\sum \K(A)^2[f_1,\ldots,f_k]$ is proper in $\K(A)$.
		\item $A\cap \sum \K(A)^2[f_1,\ldots,f_k]$ is proper in $A$.
		\item $\S(f_1,\ldots,f_k)\cap \Sp_c A\not=\emptyset$.
		\item $\S(f_1,\ldots,f_k)\cap \Sp_{pc} A\not=\emptyset$.
	\end{enumerate}

One has $(1)\implies (2)\iff (3)\iff (4)\implies (5)$.
\end{lem}

\begin{proof}
The equivalence between (2) and (3) is clear.
	We prove (2) implies (4). Assume that $\S(f_1,\ldots,f_k)\subset (\Sp_r A\setminus\Sp_c A)$. It follows that $\{\beta\in\Sp_r\K(A)\mid f_1(\beta)> 0,\ldots,f_k(\beta)> 0\}=\emptyset$.
	
Since the $f_i$ are non zero, one may equivalently say that $\{\beta\in\Sp_r\K(A)\mid f_1(\beta)\geq 0,\ldots,f_k(\beta)\geq 0\}=\emptyset$. 
	By the Positivstellensatz recalled in Theorem \ref{FormalPSS}, one gets an identity $1+p=0$ in $\K(A)$ with $p\in\sum\K(A)^2[f_1,\ldots,f_k]$. It follows that $-1\in A\cap \sum\K(A)^2[f_1,\ldots,f_k]$ and it proves that (2) implies (4) by contraposition.
	
	We prove (4) implies (2). Let $\alpha\in \S(f_1,\ldots,f_k)\cap \Sp_c A$. There exists $\beta\in \Sp_r A$ such that $\supp(\beta)=(0)$ and 
	$\beta\to\alpha$. We have $\beta\in\Sp_r \K(A)$ by Proposition \ref{orderingnull}. For $i=1,\ldots,k$, we have $f_i\in P_\alpha\setminus\supp(\alpha)$ and thus $f_i\in P_\beta\setminus\{0\}$. It follows that 
	$\sum\K(A)^2[f_1,\ldots,f_k]\subset P_\beta$ (viewed in  $\K(A)$) and thus $\sum\K(A)^2[f_1,\ldots,f_k]$ is proper.

	We prove (1) implies (2). Assume that $\sum\K(A)^2[f_1,\ldots,f_k]$ is not proper. It follows that $-1\in A\cap \sum\K(A)^2[f_1,\ldots,f_k]$ and since $-1\not\in P_\alpha$ we get that $A\cap \sum\K(A)^2[f_1,\ldots,f_k] \not\subset P_\alpha$.
	
	Since $\Sp_c A\subset\Sp_{pc} A$ then (4) implies (5).
\end{proof}

The point (2) does not necessarily imply (1) as one can see looking for example at a point distinct from the origin in the stick of the Cartan umbrella (Example \ref{CartanEx}) and $k=0$. Likewise (3) implies (1) does not hold in general, nevertheless it becomes true after quantification on the family $f_1,\ldots, f_k$ and we derive the following characterizations for central orderings:

\begin{prop} \label{charcentralgen}
	Let $\alpha\in\Sp_rA$. The following properties are equivalent:
	\begin{enumerate}
		\item $\alpha\in\Sp_cA$.
		\item For any $f_1,\ldots,f_k\in A$ such that $\alpha\in\S(f_1,\ldots,f_k)$, the cone $\sum\K(A)^2[f_1,\ldots,f_k]$ is proper in $\K(A)$.
		\item For any $f_1,\ldots,f_k\in A$ such that $\alpha\in\S(f_1,\ldots,f_k)$, the cone $A\cap\sum\K(A)^2[f_1,\ldots,f_k]$ is proper in $A$.
		\item For any $f_1,\ldots,f_k\in A$ such that $\alpha\in\S(f_1,\ldots,f_k)$, the intersection $\S(f_1,\ldots,f_k)\cap \Sp_c A$ is non-empty.
		\item For any $f_1,\ldots,f_k\in A$ such that $\alpha\in\S(f_1,\ldots,f_k)$, the cone $A\cap \sum\K(A)^2[f_1,\ldots,f_k]\subset P_\alpha$.
		\end{enumerate}
\end{prop}

\begin{proof}
The equivalence between (2), (3) and (4) is given by Lemma \ref{equivcentralbasis}.

	Let us prove that (1) implies (5). Let $\alpha\in\Sp_cA$. There exists $\beta\in \Sp_r A$ such that $\supp(\beta)=(0)$ and 
	$\beta\to\alpha$. We have $\beta\in\Sp_r \K(A)$ by Proposition \ref{orderingnull}. Let $f_1,\ldots,f_k\in P_\alpha\setminus\supp(\alpha)$ then $\forall i$, $f_i\in P_\beta\setminus\{0\}$ and it follows that 
	$\sum\K(A)^2[f_1,\ldots,f_k]\subset P_\beta$  (viewed in  $\K(A)$) and thus $\sum\K(A)^2[f_1,\ldots,f_k]$ is proper. Consequently,
	$$A\cap \sum\K(A)^2[f_1,\ldots,f_k]\subset P_{\beta}\subset P_\alpha.$$

	By Lemma \ref{equivcentralbasis}, we have that (5) implies (2).

	As already said, we know that (4) implies (1) since $\Sp_c A$ is closed for the topology 
	of $\Sp_r A$.
\end{proof}

Here appears the notion of stability index we recalled in section \ref{StabIndex}. Namely, the stability index $\sta(U)$ of an open subset $U$ of $\Sp_r A$ is the infimum of the numbers $k\in\N$ such that for any basic open subset $S$ of $\Sp_r A$ satisfying $S\subset U$ there exist $f_1,\ldots,f_k$ in $A$ with $S=\S(f_1,\ldots,f_k)$.

It leads to new characterizations for central orderings:
\begin{thm} \label{charcentral}
	Let $\alpha\in\Sp_rA$. The following properties are equivalent:
	\begin{enumerate}
		\item $\alpha\in\Sp_cA$.
		\item For any $f_1,\ldots,f_k\in A$ such that $\alpha\in\S(f_1,\ldots,f_k)$ and $k\leq \sta(\Sp_r A\setminus\Sp_c A)$, the cone $\sum\K(A)^2[f_1,\ldots,f_k]$ is proper in $\K(A)$.
		\item For any $f_1,\ldots,f_k\in A$ such that $\alpha\in\S(f_1,\ldots,f_k)$ and $k\leq \sta(\Sp_r A\setminus\Sp_c A)$, the cone $A\cap \sum\K(A)^2[f_1,\ldots,f_k]$
		is proper in $A$. 
        \item For any $f_1,\ldots,f_k\in A$ such that $\alpha\in\S(f_1,\ldots,f_k)$ and $k\leq \sta(\Sp_r A\setminus\Sp_c A)$, the intersection $\S(f_1,\ldots,f_k)\cap \Sp_c A$ is non-empty.
		\item For any $f_1,\ldots,f_k\in A$ such that $\alpha\in\S(f_1,\ldots,f_k)$ and $k\leq \sta(\Sp_r A\setminus\Sp_c A)$, the cone $A\cap \sum\K(A)^2[f_1,\ldots,f_k]\subset P_\alpha$.		
	\end{enumerate}
\end{thm}

\begin{proof}
We know that (1) is equivalent to (4) of Proposition \ref{charcentralgen}. By definition of the stability index, (4) of Proposition \ref{charcentralgen} is equivalent to (4): indeed, if $\alpha\in 
S$ with $S$ a basic open subset of $\Sp_r A$ which cannot be described by less than $\sta(\Sp_r A\setminus\Sp_c A)+1$ inequalities, then we must have 
$S\cap \Sp_c A\not=\emptyset$. By Proposition \ref{charcentralgen}, (1) implies (5). By Lemma \ref{equivcentralbasis}, (5) implies (2) and (2), (3), (4) are equivalent.
\end{proof}

Note that if the stability index $\sta(\Sp_r A\setminus\Sp_c A)$ happens to be zero and more generally if $k=0$ in Theorem \ref{charcentral}, then 
$\sum\K(A)^2[f_1,\ldots,f_k]$ reduces to 
$\sum\K(A)^2$ as the cone in $\K(A)$ generated by the empty family. In this case, assertion (5) of Theorem \ref{charcentral} is equivalent to say that $P_{\alpha}$ is precentral.

We also notice that being a precentral ordering is equivalent to satisfy condition (5) of Theorem \ref{charcentral} only for $k=0$. Then, we give similar characterizations for precentral orderings, namely, one recover conditions (2), (3) and (4) of Theorem \ref{charcentral} for $k\leq 1$.
 
\begin{thm} \label{charprecentral}
Let $\alpha\in\Sp_r A$. The following properties are equivalent:
\begin{enumerate}
\item $\alpha\in\Sp_{pc} A$.
\item For any $f\in A$ such that $\alpha\in\S(f)$, the cone $\sum\K(A)^2[f]$ is proper.
\item For any $f\in A$ such that $\alpha\in\S(f)$, the cone $A\cap\sum\K(A)^2[f]$ is proper and precentral.
\item For any $f\in A$ such that $\alpha\in\S(f)$, the intersection $\S(f)\cap \Sp_c A$ is non-empty.
\end{enumerate}
\end{thm}

\begin{proof}
The equivalence between (2), (3) and (4) are clear from Lemma \ref{equivcentralbasis}.

Assume $\alpha\not\in \Sp_{pc} A$. There exists $g\in A\cap\sum\K(A)^2$ such that $g\not\in P_\alpha$. Remark that $g\not=0$. We have $-g\in (P_\alpha\setminus\supp(\alpha))$  and thus $\alpha\in\S(-g)$. 
Suppose there exists $\beta\in(\Sp_c A\cap \S(-g))$. By definition of a central ordering and by Proposition \ref{orderingnull}, there is $\gamma\in\Sp_r A$ such that $\supp(\gamma)=(0)$ and $\gamma\to\beta$. Since $(P_\beta\setminus\supp(\beta))\subset(P_\gamma\setminus\{0\})$ then $-g\in (P_\gamma\setminus\{0\})$, it is impossible because $g\in A\cap\sum\K(A)^2\subset P_\gamma$.
We get $\S(-g)\subset (\Sp_r A\setminus\Sp_c A)$ and it proves (4) implies (1).

Assume there exists $g\in A$ such that $\alpha\in\S(g)$ and $\S(g)\subset (\Sp_r A\setminus\Sp_c A)$. We have $-g\not\in P_\alpha$. Moreover $\forall \beta\in \Sp_cA$, $-g\in P_\beta$ and thus $\forall \beta\in \Sp_cA$ such that $\supp(\beta)=(0)$, $-g\in P_\beta$. From Proposition \ref{orderingnull} and the Positivstellensatz, cf. Theorem \ref{FormalPSS}, we get $-g\in\sum\K(A)^2$. It shows that $\alpha\not\in\Sp_{pc} A$ and it proves (1) implies (4).
\end{proof}

With this characterization and the one in Theorem \ref{charcentral}, one may view precentral orderings as central orderings of "level $1$", and going further in that direction would lead to the consideration of central orderings of "level $k$". We decide not to develop such a formalism until we find some relevant applications.

The value of the stability index of the non-central locus appears to be related to the existence of a precentral ordering that is not central. Namely, one has 
$\Sp_{pc} A=\Sp_c A$ whenever $\sta(\Sp_r A\setminus\Sp_c A)\leq 1$.

In the geometric case, using Bröcker-Scheiderer Theorem and Theorem \ref{charcentral}, one gets a family of geometric rings where any precentral ordering is central :
\begin{cor} \label{centraldim2}
	Let $V$ be an irreducible affine algebraic variety over $R$ such that $\dim V\leq 2$. Then, $\Sp_{pc} R[V]=\Sp_cR[V]$.
\end{cor}

\begin{proof}
	One has $\sta(\Sp_r R[V]\setminus\Sp_c R[V])\leq \sta(\widetilde{V(R)\setminus V_{reg}(R)})$. And from Br\"ocker-Scheiderer Theorem one gets that the stability index of $V(R)\setminus V_{reg}(R)$ is at most $1$.
\end{proof}

Let us now give an example of a precentral ordering which is not central.
\begin{ex}
\label{contre-ex}
Let $V$ be the irreducible affine algebraic variety over $\R$ with coordinate ring $A=\R[V]=\R[x,y,z,t_1,t_2]/(z^2+t_1x^2+t_2y^2)$. The real part of the singular locus of $V$ is contained in the real plane in $t_1$ and $t_2$ and we are going to describe $S=V(\R)\setminus (\Cent V(\R))$ in this plane. By \cite[Prop. 7.6.2]{BCR}, $S$ is the locus of points of $V(\R)$ where the local semi-algebraic dimension is $<4$. Seeing $V(\R)$ as a variety with parameters $t_1$ and $t_2$ then we can show that 
$S=S(t_1,t_2)$, the open right upper quadrant of the plane, and the local dimension at points
of $S$ in $V(\R)$ is equal to two. 
Let us consider the four elements of the real spectrum $-1_{+\uparrow},-1_{+\downarrow},1_{+\uparrow},1_{+\downarrow}$ which, as described in \cite[Ex. 10.4.3]{BCR}, have support the ideal $(x,y,z)$ (which define the plane in $t_1,t_2$), the first two specializing to the point $(-1,0)$ and the second two specializing to the point $(1,0)$. These four orderings define a fan $F$ (see \cite[Defn. 10.4.2]{BCR}). 
It is easy to check that $F\cap \widetilde{S}=\{1_{+\uparrow}\}$.  
Since $1_{+\uparrow}\in \widetilde{S}$ then it cannot be central. 
This can also be seen algebraically by considering the property (5) of Theorem \ref{charcentral} since $t_i>0$ on $1_{+\uparrow}$ for $i=1,2$ and since $-1=t_1(x/z)^2+t_2(y/z)^2\in\sum\K(V)^2[t_1,t_2]$.

Assume now that there exists $f\in\R[V]$ such that $1_{+\uparrow}\in\S(f)$. If we assume $\S(f)\subset \widetilde{S}$, then $\#(F\cap \S(f))=1$ and by \cite[V Cor. 1.9]{ABR} we get a contradiction. It follows that $\#(F\cap \S(f))>1$. Since $F\setminus\{1_{+\uparrow}\}\subset \Sp_c A$ then using (4) of Theorem \ref{charprecentral} we get that $1_{+\uparrow}\in \Sp_{pc} A$.

To end, let us note that it is possible to give a similar example in dimension 3 by intersecting our variety with the hypersurface with equation $z-xy=0$.
\end{ex}

This example shows also that the precentral spectrum is not necessarily constructible.
Indeed, if $\Sp_{pc}\R[V]$ were a constructible subset, then, by the correspondence between semialgebraic subsets $S$ and constructible subsets $\tilde{S}$ and using Proposition \ref{closedpoints}, one would get $\Sp_{pc}\R[V]=\widetilde{\Sp_{pc} \R[V]\cap V(\R)}=\widetilde{\Sp_{c} \R[V]\cap V(\R)}=\Sp_c \R[V]$ a contradiction.

Although the central spectrum and the precentral spectrum seem to be close, it is not true that for any $\alpha\in\Sp_{pc} A$, there exists $\gamma\in\Sp_c A$ such that $\alpha\to\gamma$ as one can see using again the previous example \ref{contre-ex}.
Indeed, take $\alpha\in\Sp_{pc} A$ 
	be the ordering of all polynomial functions which are nonnegative on a $+\infty$ neigbourhood of the transcendant curve of equation $t_2=e^{-t_1}$ in the plane $(t_1,t_2)$.
	This ordering does not admit any strict specialization and it is not central since the non central locus is $\widetilde{S}(t_1 ,t_2)$.
	Moreover, arguing again with a 4-elements fan (which cannot intersect any given principal open subset at a single element), one shows that $\alpha$ is precentral.

From (4) of Theorem \ref{charprecentral}, a precentral ordering of $\Sp_r A$ is an ordering that cannot be separated from the central locus $\Sp_c A$ by principal open subsets. In the spirit of Example \ref{contre-ex}, it is possible to create precentral but non-central orderings with an higher level of non-separation with the central locus. Namely, let $k$ be an integer $\geq 2$ and let $V$ be the irreducible affine algebraic variety over $\R$ with coordinate ring $\R[V]=\R[x_1,\ldots,x_k,z,t_1,\ldots,t_k]/(z^2+t_1x_1^2+\ldots+t_kx_k^2)$. There exists 
$\alpha\in(\Sp_{pc} A\setminus\Sp_c A) $ such that for any basic open subset $S$ of $\Sp_r A$ given by $t\leq k-1$ strict inequalities then $S\cap \Sp_c A\not=\emptyset$. Moreover
we have $\alpha\in\S(t_1,\ldots,t_k)\subset (\Sp_{r} A\setminus\Sp_c A)$.

\vskip0,5cm
In the equivalent properties of Theorem \ref{charprecentral}, recall that we get rid of the condition (5) of Theorem \ref{charcentral}:
\begin{enumerate}
\item[(5)] $\forall f\in A$ such that 
$\alpha\in \S(f)$, $A\cap\sum\K(A)^2[f]\subset P_\alpha$.
\end{enumerate}

Moreover, there is another condition which arises naturally as the following one:
\begin{enumerate}
	\item[(6)] $\forall f\in A$ such that 
	$\alpha\in \S(f)$, $\exists \beta\in\Sp_c A$ with $\supp(\alpha)=\supp(\beta)$ such that $\beta\in\S(f)$.
\end{enumerate}
Indeed, by Theorem \ref{charprecentral} we know that 
$$\alpha\in\Sp_{pc} A\iff \forall f\in A\,\, {\rm such\,\, that}\,\, \alpha\in\S(f), \exists \beta\in\Sp_c A\,\, {\rm such \,\, that}\,\, \beta\in\S(f),$$
and by Propositions \ref{support3} and \ref{support4} we also know that
$$\alpha\in\Sp_{pc} A\implies \exists \beta\in\Sp_c A\,\, {\rm such\,\, that}\,\, \supp(\alpha)=\supp(\beta).$$
So, a natural question is study how $\alpha\in\Sp_{pc} A$ is related to (6).

\begin{prop}Let $\alpha\in\Sp_r A$.
Then, condition (6) is equivalent to	the following 
		\begin{center}
$\forall f\in A$ such that 	$\alpha\in \S(f)$, $\supp(\alpha)$ is $A\cap\sum\K(A)^2[f]$-convex.
		\end{center}
Moerover, one has the following implications 
	$$\alpha\in\Sp_c A\implies (5) \implies (6) \implies \alpha\in\Sp_{pc} A.$$
\end{prop}
\begin{proof}
Let us show that condition (6) implies the one of the proposition. Let $f\in A$ and $\p\in\Sp A$. Assume there exists $\beta\in\Sp_c A$ with $\p=\supp(\beta)$ such that $\beta\in\S(f)$. By Proposition \ref{support4} then $\p\in\RCent A$. There exists $\gamma$ of support $(0)$ such that $\gamma\to \beta$. We have $f(\gamma)>0$ and thus $\sum\K(A)^2[f]\subset P_{\gamma}$ in $\K(A)$. Since $\p$ is $P_{\beta}$-convex 
then $\p$ is convex for $P_{\gamma}$ and hence also for $A\cap\sum\K(A)^2[f]$. 

Conversely, take $\alpha\in\Sp_r A$ such that
$\supp(\alpha)=\p$, $\alpha\in \S(f)$ and $\p$ is $A\cap\sum\K(A)^2[f]$-convex. By \cite[Prop. 4.2.9]{BCR} there exists an ordering $\gamma\in\Sp_r\K(A)$ such that 
$\p$ is $P_{\gamma}\cap A$-convex. By Propositions \ref{orderingnull} and \ref{specialisation} there exists $\beta\in\Sp_r A$ such that $P_{\gamma}\cap A\to P_{\beta}$ and $\supp(\beta)=\p$. Clearly $\beta\in\Sp_c A$ and since $A\cap\sum\K(A)^2[f]\subset P_{\beta}$ then $f(\beta)\geq 0$. Since $\alpha\in \S(f)$ and $\supp(\beta)=\supp(\alpha)=\p$ then $\beta\in \S(f)$. We have shown the first assertion.

Let us now prove the implications.
The first one comes directly from characterizations of Proposition \ref{charcentralgen}.
The second one relies on the fact that $\supp(\alpha)$ is $P_\alpha$-convex and if
$A\cap\sum\K(A)^2[f])\subset P_\alpha$, then 
$\supp(\alpha)$ is $A\cap\sum\K(A)^2[f]$-convex.
And the last implication is also immediate using Theorem \ref{charprecentral}. 
\end{proof}

Example \ref{contre-ex} allows us to study converse implications. 
Namely, the precentral ordering $\alpha=1_{+\uparrow}$ considered there satisfies condition (6) (since $1_{+\downarrow}\in\Sp_c \RR[V]$) but not condition (5).
Indeed, one has identity $-1=t_1(x/z)^2+t_2(y/z)^2\in\sum\K(V)^2[t_1,t_2]$. Multiplying by $t_1$, one gets $-t_1\in\RR [V]\cap\sum\K(V)^2[t_1t_2]$ but
$\alpha\in\S(t_1t_2)$ and $-t_1\not\in P_\alpha$.

It happens also that the converse of the first implication doesn't hold. Indeed, consider the slighly modified example:
$A=\R[V]=\R[x_1,x_2,x_3,z,t_1,t_2,t_3]/(z^2+t_1x_1^2+t_2x_2^2+t_3 x_3^2)$.
Let us take $\beta\in\Sp A\setminus \Sp_c A=\S(t_1,t_2,t_3)$ such that $\beta\in\Sp_{pc}A$. Then, any basic open defined by two or less inequalities which contains $\beta$ has to intersect $\Sp_c A$.
Let us argue by contradiction and assume that $\beta$ does not satisfy property (5). Then, there is $f\in A$ such that $\beta\in\S(f)$
and $A\cap\sum \K(A)^2[f]\not\subset P_\beta$. One has an identity $a=s+ft\notin P_\beta$ with $s,t\in\sum \K(A)^2$, namely 
$\beta\in \S(f,-a)$ and hence 
$\S(f,-a)$ intersects $\Sp_cA$. Let 
$\beta'\in\S(f,-a)\cap \Sp_c A$, there exists $\gamma$ of support $(0)$ such that $\gamma\to \beta'$.
Relatively to the ordering $\gamma$, one gets $s+ft> 0$ whereas $a<0$, a contradiction. Hence $\beta$ satisfies (5) and gives a counterexample to the first implication.
And concerning the converse of the last implication, we did not succeed to prove or disprove it.

To end this section whose aim was to compare central and precentral orderings, we recall that a central domain is obviously precentral and we prove that the converse is also true in the geometric case.
\begin{prop}
\label{precentimpcentvar}
Let $V$ be an irreducible affine algebraic variety over $R$. The following properties are equivalent:
\begin{enumerate}
\item $V$ is central.
\item $R[V]$ is central.
\item $R[V]$ is precentral.
\end{enumerate}
\end{prop}

\begin{proof}
By Proposition 
\ref{alg+geomcentrality} 
we are left to prove (3) implies (1). Assume $R[V]$ is precentral. We have $\Sp_r R[V]=\Sp_{pc} R[V]=\widetilde{V(R)}$ and thus $\Sp_{pc} R[V]\cap V(R)=V(R)$. Using Proposition \ref{closedpoints} then $\Sp_{pc} R[V]\cap V(R)=\Sp_{c} R[V]\cap V(R)$ and the proof is done.
\end{proof}

\section{Applications: Central and precentral Positivstellens\"atze} 

Unless otherwise stated, $A$ is a domain with fraction field $\K(A)$ which is  formally real.


\subsection{Around the Hilbert 17th problem}

Real geometers have always tried to give certificates of positivity for different types of functions. The most famous of these certificates is undoubtedly that of the 17th problem. Hilbert thus wanted to characterize the polynomials of $A=R[x_1,\ldots,x_n]$ which are nonnegative on $R^n$. 
Since such a polynomial is not necessarily a sum of squares in $R[x_1,\ldots,x_n]$ which one may reformulate saying that, if $n\geq 2$, the intersection of all cones in $A$ is strictly contained in the intersection of all orderings in $A$, namely $$\bigcap_{P\in\CO(A)}P\subsetneq\bigcap_{P\in\Sp_r A} P.$$ 
The Artin answer to Hilbert 17th problem says that a nonnegative polynomial is a sum of squares in $\K(A)=R(x_1,\ldots,x_n)$ : the trace on $A$ of the intersection of all cones of $\K(A)$ coincide with the intersection of all orderings in $A$, namely:
$$A\cap\bigcap_{P\in\CO(\K(A))}P=A\cap\bigcap_{P\in\Sp_r \K(A)} P=\bigcap_{P\in\Sp_r A} P.$$

We are interested by getting central Hilbert 17th properties for a general domain $A$, namely finding what kind of positivity is given by the algebraic certificate of an $f$ which belongs to the cone $\C=A\cap\sum \K(A)^2$.

In this direction, the classical Hilbert 17th property can be reformulated with the language of central cones and orderings of $A$. Since $A=R[x_1,\ldots,x_n]$ is a central and precentral ring, 
the following sequence of inclusions 
$$\C=\bigcap_{P\in\CO_{c}(A)}P\subset \bigcap_{P\in\Sp_{pc} A} P\subset \bigcap_{P\in\Sp_c A} P,$$ are in fact equalities.

In the sequel we give some central or precentral certificates of vanishing (Nullstellens\"atze) and of positivity (Positivstellens\"atze) on several subsets of $\Sp_r A$.

\subsection{Central and precentral Nullstellens\"atze}

We start by studying certificates of vanishing. 
Let $I$ be an ideal of $A$. Denote by $\Zi^r(I)$ the set of all $\alpha\in\Sp_r A$ such that $I\subset \supp (\alpha)$. Then we write $\Zi^c(I)=\Zi^r(I)\cap \Sp_c A$ and $\Zi^{pc}(I)=\Zi^r(I)\cap \Sp_{pc} A$. 
Let $W\subset\Sp_r A$, we denote by $\I(W)$ the set of $f\in A$ such that $W\subset \Zi^r (f)$; it is clearly an ideal of $A$.

Note that the inclusion $\Zi^c(I)\subset \Zi^{pc}(I)$ can be strict, take $I=(0)$ in Example \ref{contre-ex}.
One may nevertheless go further to get an abstract central and precentral Nullstellens\"atze:
\begin{prop}\label{CentralNSS} (Central and precentral Nullstellens\"atze):\\
	Let $I$ be an ideal of $A$. One has $$\I(\Zi^c(I))=\I(\Zi^{pc}(I))=\RadCe{I}.$$
\end{prop}

\begin{proof}
Since $\Sp_c A\subset\Sp_{pc} A$ then $\Zi^c(I)\subset \Zi^{pc}(I)$ and thus $\I(\Zi^c(I))\supset\I(\Zi^{pc}(I)))$.

Clearly $\RadCe{I}\subset \I(\Zi^{pc}(I))$.
Indeed, let $a\in \RadCe{I}$ . We have 
	$a^{2m}+b\in I$ with $b\in \C$.
	Let $\alpha\in \Zi^{pc}(I)$: one has
	$\supp(\alpha)\supset I$, and hence 
	$a^{2m}+b\in \supp(\alpha)$. By Proposition \ref{support3} we have $\supp(\alpha)\in\RCent A$. A central ideal is $\C$-convex and we get $a^{2m}\in \supp(\alpha)$.
	By radicality of $\supp(\alpha)$ then $a\in\supp(\alpha)$.
	
	Let $f\in \I(\Zi^c(I))$. Let us show that, if $\p$ is a central prime ideal containing $I$, then $f\in \p$. By \cite[Prop. 3.14]{Mnew} we will then get $f\in\RadCe I$.
	Using Proposition \ref{support4} there exists $\alpha\in\Sp_c A$ such that $\supp (\alpha)=\p$. Clearly $\alpha\in\Zi^c(I)$ and thus $f(\alpha)=0$ i.e $f\in\supp(\alpha)=\p$.
	We get $	\I(\Zi^c(I)\subset \RadCe I$ and it ends the proof.
\end{proof}

\subsection{Precentral Positivstellens\"atze}

One may carry on further to get central and precentral Positivstellens\"atze having in mind that the algebraic nature of precentrality seems much more convenient than the geometric nature of centrality.

To get geometric central Positivstellens\"atze, our strategy is to establish first abstract precentral Positivstellens\"atze and then derive some abstract central ones and finally
get geometric central ones.

A set of the form $\overline{\S}(f_1,\ldots,f_k)=\{ \alpha\in\Sp_r A\mid f_1(\alpha)\geq 0,\ldots, f_k(\alpha)\geq 0\}$
for $f_1,\ldots,f_k$ some elements of $A$,  is called a basic closed subset of $\Sp_r A$. We denote by $\S^c(f_1,\ldots,f_k)$, $\overline{\S}^c(f_1,\ldots,f_k)$, $\S^{pc}(f_1,\ldots,f_k)$ and $\overline{\S}^{pc}(f_1,\ldots,f_k)$, the sets $\S(f_1,\ldots,f_k)$, $\overline{\S}(f_1,\ldots,f_k)$ intersected respectively with $\Sp_c A$ and $\Sp_{pc} A$.

With these notations, one gets:

\begin{thm}\label{pc-PSS} (Precentral Positivstellens\"atze):\\
	Let $f_1,\ldots,f_r$ in $A$ and $f\in A$. One has:
	\begin{enumerate}
\item $f\geq 0$ on $\overline{\S}^{pc}(f_1,\ldots,f_r)$ if and only if $$fq=p+f^{2m}$$ where $p,q$ are in $\C[f_1,\ldots,f_r]$.
\item $f>0$ on $\overline{\S}^{pc}(f_1,\ldots,f_r)$ if and only if $$fq=1+p$$ where $p,q$ are in $\C[f_1,\ldots,f_r]$.		
\item $f=0$ on $\overline{\S}^{pc}(f_1,\ldots,f_r)$ if and only if $$f^{2m}+p=0$$
where $p$ is in $\C[f_1,\ldots,f_r]$.		
	\end{enumerate}	
\end{thm}

\begin{proof}
	\begin{enumerate}
	\item Let $H$ be the set $\C\cup\{f_1,\ldots,f_r,-f\}$ and $M$ the monoid generated by $f$.	
	Then, there is no $\alpha\in 
	\Sp_{pc}A$ such that $-f(\alpha)\geq 0$, $f_1(\alpha)\geq 0,\ldots,f_r(\alpha)\geq 0$  and $f(\alpha)\not=0$.
	if and only if there is no $\alpha\in 
	\Sp_{r}A$ such that 
	$H\subset \alpha$
	and $-f(\alpha)\not=0$.

Then, from the formal Positivstellensatz recalled in Theorem \ref{FormalPSS}, it is equivalent to have an identity of the form
$p-fq+f^{2m}=0$ where $p,q\in \C[f_1,\ldots,f_r]$.

	\item Let $H$ be the set $\C\cup\{f_1,\ldots,f_r,-f\}$ and $M$ the monoid generated by $1$.	
There is no $\alpha\in \Sp_{pc} A$ such that $f_1(\alpha)\geq 0,\ldots,f_r(\alpha)\geq 0$, $-f(\alpha)\geq 0$ and $1(\alpha)\not=0$ 
if and only if and only if there is no $\alpha\in 
	\Sp_{r}A$ such that 
	$H\subset \alpha$
	and $1(\alpha)\not=0$ and conclude using the formal Positivstellensatz, cf. Theorem \ref{FormalPSS}.
			
	\item Let $H$ be the set $\C\cup\{f_1,\ldots,f_r\}$ and $M$ the monoid generated by $f$. Since there is no $\alpha\in \Sp_{pc} A$ such that $f_1(\alpha)\geq 0,\ldots,f_r(\alpha)\geq 0$ and 
	$f(\alpha)\not =0$ if and only if there is no $\alpha\in 
	\Sp_{r}A$ such that 
	$H\subset \alpha$
	and $f(\alpha)\not=0$, we get the proof using again the formal Positivstellensatz, cf. Theorem \ref{FormalPSS}.	
\end{enumerate}		
\end{proof}

As a particular case, one gets an asbtract precentral Hilbert 17th property:

\begin{thm}\label{precentralHilbert17} (Precentral Hilbert 17th property):\\
	Let $f\in A$. The following properties are equivalent:
	\begin{enumerate}
\item $f\geq 0$ on $\Sp_{pc} A$.
\item There exist $p,q$ in $\C$ such that $fq=p+f^{2m}$. 
\item There exist $p,q\in\C$ such that $q^2 f=p$ and $\Zi^{pc}(q)\subset\Zi^{pc}(f)$.
\item $f\in\C$.
\end{enumerate}
\end{thm}

\begin{proof}
Les us show (1) $\Rightarrow$ (2). 	
Assume $f\geq 0$ on $\Sp_{pc} A$. Since  $f\geq 0$ on $\overline{\S}^{pc}(1)$ then by Theorem \ref{pc-PSS} one gets (2).

Let us show (2) implies (3).
Assume $fq=p+f^{2m}$ with $p,q\in\C$. One may assume that $q\not=0$ or equivalently that $p+f^{2m}\not=0$ since otherwise $f=0$ (by hypothesis $\K(A)$ is formally real and hence the null ideal is real in $A$). One gets
	$f(p+f^{2m})=f^2q$ which gives 
	$f=\frac{(f^2q)(p+f^{2m})}{(p+f^{2m})^2}\in \C$. Set $s=(f^2q)(p+f^{2m})$ and $t=p+f^{2m}$. We have $t^2f=s$ and 
$s,t\in\C$. We have $\Zi^{pc}(t)\subset \Zi^{pc}(f)$: if $\alpha\in\Zi^{pc}(p+f^{2m})$ then $p+f^{2m}\in\supp(\alpha)$ which is a central ideal and by definition we get $f\in\supp(\alpha)$. This shows the desired implication.

Trivially (3) implies (4).
	
To end, let us show that (4) implies (1).	
Assume $f\in\C$ then it is clear from the definition of a precentral cone that $f\geq 0$ on $\Sp_{pc} A$.
\end{proof}

Hence, in any domain $A$ with formally real fraction field, the intersection of all central cones coincides with the intersection of all precentral orderings.

\subsection{Central Positivstellens\"atze}
\smallskip

From the precentral Positivstellens\"atze, one may deduce some central ones. Let us remind first that Proposition \ref{CentralNSS} gives that $f=0$ on $\Zi^{pc}(I)$ if and only if $f=0$ on $\Zi^{c}(I)$. One also have :

\begin{lem}\label{posprecent}
	\begin{enumerate}
		\item 	$f>0$ on $\Sp_{pc}A$ if and only if 
		$f>0$ on $\Sp_{c}A$. 
		\item 	$f\geq 0$ on $\Sp_{pc}A$ if and only if 
		$f\geq 0$ on $\Sp_{c}A$. 		
		\end{enumerate}
\end{lem}

\begin{proof}
The direct implications are clear since $\Sp_c A\subset \Sp_{pc} A$.

Let $f>0$ on $\Sp_{c}A$. Then, $f>0$ on $\Sp_{r}\K(A)$ and hence $f\in \sum \K(A)^2$.
Hence, $f\geq 0$ on $\Sp_{pc}A$. If there exists $\alpha\in \Sp_{pc}A$ such that $f(\alpha)=0$ then $f\in \supp(\alpha)$ and by Proposition \ref{support3} then $\supp(\alpha)$ is a central ideal. By Proposition \ref{support4}, there is $\beta\in \Sp_{c}A$ such that $\supp(\alpha)=\supp(\beta)$ and thus $f(\beta)=0$, a contradiction.
	
Let $f\geq 0$ on $\Sp_{c}A$. Assume that $f(\beta)<0$ with $\beta$ a precentral ordering. By characterization of Theorem \ref{charprecentral}, one gets the existence of $\gamma$ central such that $f(\gamma)<0$, contradiction.	
\end{proof}

On the other hand, beware that 
in general $\S^{pc}(f)\not =\S^{c}(f)$ as one can see using Example \ref{contre-ex} with $f=t_1$.
Likewise, beware that in general $\overline{\S}^{pc}(f)\not =\overline{\S}^c(f)$. Indeed, if we work now in the domain $B=A[t]/(tt_1-1)$ where $A$ is the domain of Example \ref{contre-ex}, one has also an element $g=-t_2$ such that $g\geq 0$ on $\overline{\S}^{c}(f)$ but not 
on $\overline{\S}^{pc}(f)$. These observations shows that it is possible to derive a Hilbert 17th property from the precentral one, although it is 
not possible to derive central Positivstellens\"atze from Theorem \ref{pc-PSS}.

\begin{prop}\label{centralHilbert17} (Hilbert 17th Property):\\
Let $f\in A$. The following properties are equivalent:
	\begin{enumerate}
\item $f\geq 0$ on $\Sp_{c} A$.
\item There exist $p,q$ in $\C$ such that $fq=p+f^{2m}$. 
\item There exist $p,q\in\C$ such that $q^2 f=p$ and $\Zi^{c}(q)\subset\Zi^{c}(f)$.
\item $f\in\C$.
\end{enumerate}
\end{prop}

\begin{proof}
Using Lemma \ref{posprecent} and Theorem \ref{precentralHilbert17}, we see that (1), (2) and (4) are equivalent. Clearly (3) implies (4). 

Let us show that (2) implies (3). Assume $fq=p+f^{2m}$ with $p,q\in\C$. Repeating the arguments used in the proof of (2) implies (3) of Theorem \ref{precentralHilbert17} we get an identity $t^2f=s$ with 
$s,t\in\C$ and $\Zi^{pc}(t)\subset \Zi^{pc}(f)$ and it is easy to see that we also have $\Zi^{c}(t)\subset \Zi^{c}(f)$.
\end{proof}

Note that $\Sp_c A$ is the smallest closed subset of $\Sp_r A$ such that the nonnegativity on this subset is equivalent to be in $\C$.

One may also give an interpretation with cones, namely : in any domain with formally real fraction field the intersection of all central cones coincides with the intersection of all central orderings.

Although we do not obtain central Positivstellens\"atze in the general case, under a condition on the stability index, the central and precentral spectra coincide 
and one gets from Theorem \ref{pc-PSS}:

\begin{prop}\label{c-PSSSurf} (Central Positivstellens\"atze in low dimension):\\
	Assume that $\sta (\Sp_r A\setminus \Sp_c A)\leq 1$.
	Let $f_1,\ldots,f_r$ in $A$ and $f\in A$.
	One has:
	\begin{enumerate}
\item $f\geq 0$ on $\overline{\S}^{c}(f_1,\ldots,f_r)$ if and only if $$fq=p+f^{2m}$$ where $p,q$ are in $\C[f_1,\ldots,f_r]$.
\item $f>0$ on $\overline{\S}^{c}(f_1,\ldots,f_r)$ if and only if $$fq=1+p$$ where $p,q$ are in $\C[f_1,\ldots,f_r]$.		
\item $f=0$ on $\overline{\S}^{c}(f_1,\ldots,f_r)$ if and only if $$f^{2m}+p=0$$
where $p$ is in $\C[f_1,\ldots,f_r]$.			\end{enumerate}	
\end{prop}

\vskip 0,2cm\noindent\textbf{Remark}
\begin{enumerate}
\item[(i)] Assertion (1) is false in the case $\sta(\Sp_r A\setminus \Sp_c A)=2$. Consider the domain $B=A[t]/(tt_2-1)$ where $A$ is the ring in Example \ref{contre-ex}. As previously noticed,
working in $\Sp_r B$, $-t_2\geq 0$ on $\overline{\S}^{c}(t_1)$ but not 
on $\overline{\S}^{pc}(t_1)$. By Theorem \ref{pc-PSS}, we cannot have an identity of the form $-t_2q=p+(t_2)^{2m}$ where $p,q$ are in $\C[t_1]$.
\item[(ii)] Likewise, using Example \ref{contre-ex}, assertion (3) is false in the case $\sta(\Sp_r A\setminus \Sp_c A)=2$. Indeed, $t_1t_2=0$ on $\overline{\S}^{c}(t_1,t_2)$ but not on $\overline{\S}^{pc}(t_1,t_2)$.
\item[(iii)] One may give a counter example to (2) in the case $\sta(\Sp_r A\setminus \Sp_c A)=2$..
Indeed, let us consider the domain $B=A[t,s]/(tt_1-1, st_2-1)$ where $A$ is the ring in Example \ref{contre-ex}. It can be shown that $-t_2>0$ on $\overline{\S}^{c}(t_1)$ but not 
on $\overline{\S}^{pc}(t_1)$.
\item[(iv)] In the full case i.e if $\overline{\S}^{c}(f_1,\ldots,f_r)=\Sp_c A$ then the assertions (1) and (2) of the proposition are valid without assumption on the stability index by Proposition \ref{centralHilbert17}.
\end{enumerate}
\vskip 0,2cm

Besides, without assumption on the stability index, assertion (3) is always valid whenever we consider a basic closed subset with a single inequality.
\begin{prop}\label{c-PSSSurfOneIneq}
	Let $f,g$ in $A$.
	Then, 
	\begin{enumerate} 
	\item[] $f=0$ on $\overline{\S}^{c}(g)$ if and only if $f^{2m}+p=0$ where $p$ is in $\C[g]$.	
	\end{enumerate}	
\end{prop}

\begin{proof}
	Assume that $f=0$ on $\overline{\S}^{c}(g)$. 
	
	Let $\alpha\in \overline{\S}^{pc}(g)$. 
	If $g(\alpha)=0$, then $g$ belongs to the support of $\alpha$ which is known to be also the support of a central ordering $\beta$. 
	By assumption $f(\beta)=0$ and hence $f(\alpha)=0$.
	
	If $g(\alpha)>0$, then there is a central ordering $\beta'$ such that $g(\beta')>0$ by (4) of Theorem \ref{charprecentral}. Hence, there is an ordering $\gamma$ of support $(0)$ such that $\gamma\to\beta'$ and thus $g(\gamma)>0$.
	By assumption $f(\gamma)=0$ and hence $f=0$, in particular $f(\alpha)=0$.
	
	It follows that $f=0$ on $\overline{\S}^{pc}(g)$ and we get the proof by (3) of Theorem \ref{pc-PSS}.
\end{proof}

\subsection{Geometric central Positivstellens\"atze}

Let us now write down the cases of geometric rings.

These are obtained from the central abstract results of the previous subsection together with the so-called Artin-Lang property (cf \cite[Thm. 4.1.2]{BCR}). 

We first give a slightly more detailed version of \cite[Thm. 6.1.9]{BCR}.

\begin{prop} \label{geomcentralHilbert17} (Geometric Hilbert 17th Property):\\
Let $V$ be an irreducible algebraic variety over $R$. Let $f\in R[V]$ and $\C=R[V]\cap\sum \K(V)^2$.
The following properties are equivalent:
	\begin{enumerate}
\item $f\geq 0$ on $\Cent V(R)$.
\item There exist $p,q$ in $\C$ such that $fq=p+f^{2m}$. 
\item There exist $p,q\in\C$ such that $q^2 f=p$ and $\Zi(q)\cap \Cent V(R)\subset\Zi(f)\cap \Cent V(R)$.
\item $f\in\C$.
\end{enumerate}
\end{prop}

\begin{proof}
By the Artin-Lang property (cf \cite[Thm. 4.1.2]{BCR}) or alternatively by Tarski-Seidenberg property, if $f\geq 0$ on $\Cent V(R)$, then $f\geq 0$ on $\Sp_{c} R[V]$ since $\widetilde{\Cent V(R)}=\Sp_{c} R[V]$. Remark also that for $g\in R[V]$ we have $\widetilde{\Zi(g)\cap \Cent V(R)}=\Zi^{c} (g)$.
One may then use Proposition \ref{centralHilbert17}.	
\end{proof}

One also has:
\begin{prop}\label{geoc-PSSSurf} (Geometric central Positivstellens\"atze for surfaces):\\
Let $V$ be an irreducible algebraic variety over $R$ such that $\dim (V(R))\leq 2$. Let $f,f_1,\ldots,f_r$ in $R[V]$ and $\C=R[V]\cap\sum \K(V)^2$.
	One has:
	\begin{enumerate}
\item $f\geq 0$ on $\overline{\S}(f_1,\ldots,f_r)\cap \Cent V(R)$ if and only if $$fq=p+f^{2m}$$ where $p,q$ are in $\C[f_1,\ldots,f_r]$.
\item $f>0$ on $\overline{\S}(f_1,\ldots,f_r)\cap \Cent V(R)$ if and only if $$fq=1+p$$ where $p,q$ are in $\C[f_1,\ldots,f_r]$.		
\item $f=0$ on $\overline{\S}(f_1,\ldots,f_r)\cap \Cent V(R)$ if and only if $$f^{2m}+p=0$$
where $p$ is in $\C[f_1,\ldots,f_r]$.		
	\end{enumerate}	
\end{prop}

\begin{proof}

Since $\dim (V(R))\leq 2$, one has $\sta(\Sp_r R[V]\setminus \Sp_c R[V])\leq 1$.

Let us show just (1) since one proceeds likewise for the other properties.
We show the non obvious implication. Let us assume that $f\geq 0$ on $\overline{\S}(f_1,\ldots,f_r)\cap \Cent V(R)$.
By Artin-Lang property \cite[Thm. 4.1.2]{BCR}, one deduces that 
	$f\geq 0$ on $\overline{\S}^{c}(f_1,\ldots,f_r)$. 
	One concludes then by application of
	(1) Proposition \ref{c-PSSSurf}.	
\end{proof}

\subsubsection{In the Nash setting} 

Recall that a Nash function on $R^n$ is a semialgebraic function of class $C^{+\infty}$ (typically $\sqrt{1+x^2}$ is a Nash function on $R$). 
Let us denote by $\mathcal{N}(R^n)$ the ring of all Nash functions on $R^n$.
Let us consider an irreducible Nash set $V=\Zi(I)$ given by a prime ideal $I\subset \mathcal{N}(R^n)$. 
Let us denote by $A$ or $\mathcal{N}(V)$ the quotient ring $\mathcal{N}(R^n)/I$ which can be seen as the ring of Nash functions over $V$.
This ring is an excellent ring as one can see using the same argument than in the proof of \cite[VIII Prop. 8.4]{ABR}, namely the criterion stated in \cite[VII Prop. 2.4]{ABR}.

As for the polynomials, one may define the central locus $\Cent(V)$ of $V$ as the Euclidean closure of the set $\Reg(V)$ of Nash regular points of $V$ (and its coincides with the algebraic central locus when $V$ is algebraic). Namely, a point $x\in V$ associated to the maximal ideal $m_x$ is said to be regular if the local ring $A_{m_x}$ is regular.

Note that since a Nash function is semialgebraic, it gives sense to $\widetilde{\Cent(V)}\subset \widetilde{R^n}$.
Moreover, recall from 
\cite[Prop. 8.8.1]{BCR} that the canonical morphism
$\Sp_r\mathcal{N}(R^n)\to \Sp_r R[x_1,\ldots,x_n]=\widetilde{R^n}$ 
is an homeomorphism and induces another homeomorphism $$\Sp_r \mathcal{N} (V)\simeq \widetilde{V}.$$

One key tool of the polynomial case to relate geometry to algebra is the tilde operator. For instance, we have already seen that it commutes with the topological closure. Namely, for any semialgebraic subset $S$ of $R^n$, by \cite[Thm. 7.2.3]{BCR} one has 
$\widetilde{\overline{S}^E}=\overline {\widetilde{S}}$.
Roughly speaking, this commutation is still valid in the Nash case:

\begin{lem}\label{CentralSpectrumNash}
	We have
$\widetilde{\Cent(V)}=\Sp_c \mathcal{N}(V)$.
\end{lem}
\begin{proof}
	As previously recalled, we see both quantities $\widetilde{\Cent(V)}$ and $\Sp_c \mathcal{N}(V)$ in $\widetilde{R^n}=\Sp_r R[x_1,\ldots,x_n]$.

At the Zariski spectrum level, one has $\Sing(\mathcal{N}(V))=\V(J)$, whereas
at the geometrical level on a has $\Sing(V)=\Zi(J)$, where $J$ is an ideal of $\mathcal{N}(R^n)$ containing $I$.
Hence, at the real spectrum level, one gets $\widetilde{\Sing(\mathcal{N}(V))}=\widetilde{\Sing(V)}$ where the first tilde send a Zariski constructible subset of $\Sp \mathcal{N}(V)$ to a constructible subset of $\widetilde{R^n}$ (see just before Proposition \ref{abstract764}) and the second tilde is the usual one from $R^n$ to $\widetilde{R^n}$.

Then, one derives $\widetilde{\Reg(\mathcal{N}(V))}=\widetilde{\Reg(V)}$
and 
$\overline{\widetilde{\Reg(\mathcal{N}(V))}}=\overline{\widetilde{\Reg(V)}}=\widetilde{\overline{\Reg(V)}^E}$.
Using Proposition \ref{abstract764} we get $\Sp_c\mathcal{N}(V)=\widetilde{\Cent(V)}$. 
\end{proof}

Then, one gets a similar statement than in the algebraic case. 

\begin{prop} \label{geomcentralNashHilbert17} (Central Nash Hilbert 17th Property):\\
Let $V$ be an irreducible Nash set. Let $f\in A=\mathcal{N}(V)$ and $\C=A\cap \sum \K(A)^2$.
The following properties are equivalent:
	\begin{enumerate}
\item $f\geq 0$ on $\Cent V$.
\item There exist $p,q$ in $\C$ such that $fq=p+f^{2m}$. 
\item There exist $p,q\in\C$ such that $q^2 f=p$ and $(\Zi(q)\cap \Cent V)\subset(\Zi(f)\cap \Cent V)$.
\item $f\in\C$.
\end{enumerate}
\end{prop}

\begin{proof}
Saying $f\geq 0$ on $\Cent V$ is equivalent to say that the semialgebraic subset $S=\{f<0\}\cap \Cent V$ of $R^n$ is empty.
By the Artin-Lang property (cf \cite[Thm. 4.1.2]{BCR}) and Lemma \ref{CentralSpectrumNash}, it is equivalent to the emptiness of $\widetilde{S}=\S(-f)\cap \Sp_c A=\S^c (-f)\subset \Sp_r A$ i.e $f\geq 0$ on $\Sp_c A$.
One may conclude using Proposition \ref{centralHilbert17}.	
\end{proof}

\begin{prop}\label{Nashalypc-PSS}(Nash Central Positivstellens\"atze for surfaces):\\
	Let $f,f_1,\ldots,f_r$ in $A=\mathcal{N}(V)$ the ring of Nash functions  on an irreducible Nash set $V$ of dimension $\leq 2$.	
	Let $\C=A\cap \sum \K(A)^2$. Then,
	\begin{enumerate}
		\item $f\geq 0$ on $\overline{\S}(f_1,\ldots,f_r)\cap \Cent(V)$ if and only if $fq=p+f^{2m}$ where $p,q$ are in $\C[f_1,\ldots,f_r]$.
		\item $f>0$ on $\overline{\S}(f_1,\ldots,f_r)\cap \Cent(V)$ if and only if $fq=1+p$ where $p,q$ are in $\C[f_1,\ldots,f_r]$.		
		\item $f=0$ on $\overline{\S}(f_1,\ldots,f_r)\cap \Cent(V)$ if and only if $f^{2m}+p=0$
		where $p$ is in $\C[f_1,\ldots,f_r]$.		
	\end{enumerate}	
\end{prop}

\begin{proof}
To use the general framework we have introduced, we have first to show that the stability index corresponds to the dimension also in the Nash setting. We use the Artin-Mazur description of Nash functions (cf \cite[Thm. 8.4.4]{BCR}) which states that for any Nash functions $f_1,\ldots,f_r:R^n\to R$ there is a nonsingular algebraic set $X\subset R^q$ of dimension $n$, an open semi-algebraic subset $W$ of $X$, a Nash diffeomorphism $\sigma:R^n\to W$ and some polynomial function $g_1,\ldots,g_r$ on $W$ such that $f_i=g_i\circ \sigma$. Hence, any description of a basic open subset $\{f_1>0,\ldots,f_r>0\}$ in $V\subset R^n$ where the $f_i$'s are Nash functions can be translated via $\sigma$ into a basic open subset $\{g_1>0,\ldots,g_r>0\}$ in $W$ where the $g_i$'s are polynomial functions. Then, one may apply the Br\"ocker Scheiderer theorem for the stability index of algebraic varieties.

	Let us show now the first assertion, one proceeds likewise for the others.
	
	Saying that $f\geq 0$ on $\overline{\S}(f_1,\ldots,f_r)\cap \Cent(V)$ means 
	that the semialgebraic subset 
	$S=\{f<0,f_1\geq 0,\ldots, f_r\geq 0\}\cap \Cent(V)$ of $V$ is empty.
	By the Artin-Lang property (cf \cite[Thm. 4.1.2]{BCR}) and Lemma \ref{CentralSpectrumNash}, it is equivalent to the emptiness of $\widetilde{S}=\S^c(-f)\cap \overline{\S}^c(f_1,\ldots,f_r)\subset \Sp_r A$.
	It means $f\geq 0$ on $ \overline{\S}^c(f_1,\ldots,f_r)$.
	We may then use Proposition \ref{c-PSSSurf} to conclude.
\end{proof}

\subsubsection{In the analytic setting}
Less classical than in the real algebraic setting, one may define the central locus of an irreducible real analytic set as the euclidean closure of the set of all regular points.

Let $\Omega$ be a real analytic variety and $C$ be a compact global semianalytic subset of $\Omega$.
Let $\mathcal{O}(\Omega_C)$ be the ring of all germs of all real analytic functions at $C$ ; it is a noetherian ring. For $X_C$ a semianalytic germ of $\Omega_C$, we set 
$A=\mathcal{O}(X_C)=\mathcal{O}(\Omega_C)/\mathcal{I}(X_C)$ to be the ring of analytic function germs of $X_C$ where $\mathcal{I}(X_C)$ is the ideal of functions germs vanishing at 
$X_C$. In all the following, we consider $X_C$ irreducible, meaning that $A$ is a domain and $\K(A)$ is the field of germs of meromorphic functions. 

We are mainly interested in germs of analytic functions at a point and also by the analytic functions on a compact subset both cases covered by our framework.
Note that without the compactness assumption, the ring $A$ would not be so nice, for instance not noetherian.

From now on, we consider that $X_C$ is a subanalytic set germ of $\Omega_C$.
For $x\in X_C$, we say that $x$ is regular if the ring $\mathcal{O}(X_C)_{m_x}$ is regular where $m_x$ is the ideal associated to $x$.
One may then define $\Cent(X_C)$ the analytic central locus germ of $X_C$ to be the euclidean closure of the set of regular points in $X_C$. 

Since $C$ is compact, by \cite[VIII Thm. 7.2]{ABR}, $\Cent(X_C)$ is a closed semianalytic subset germ of $X_C$ and it can be defined by a \textit{finite} of a conjonction of inequalities. 
 
Again since $C$ is compact, one may use the key tool taken from \cite[VIII Prop. 8.2]{ABR} that we recall for the convenience of the reader:

\vskip 0,2cm\noindent\textbf{Proposition:}
The tilde correspondance 
$$\cup_i\{f_{i1}>0,\ldots f_{ir_i}>0,g_{i1}=0,\ldots,g_{is_i}=0\}\mapsto
\cup_i\{\alpha\in \widetilde{X_C\mid }f_{i1}(\alpha)>0,\ldots$$
$$\ldots f_{ir_i}(\alpha)>0,g_{i1}(\alpha)=0,\ldots,g_{is_i}(\alpha)=0\}
$$
induces an isomorphism between the boolean algebra of semianalytic subset germs of $X_C$ onto that of constructible subsets of $\widetilde{X_C}=\Sp_r \mathcal{O}(X_C)$.
\vskip 0,2cm

Beware that without the compactness hypothesis, the tilde operation is no more well defined and we only have a weak Artin-Lang property.

From this, one may derive the counterpart of the classical properties on the tilde operator between semialgebraic subsets of $\RR^n$ and constructible subsets of $\widetilde{\RR^n}$. 
One also gets an Artin-Lang property, namely $S=\emptyset$ if and only if $\widetilde{S}=\emptyset$ for sets $S$ as described in the proposition. 

Let us write down the compatibility of the tilde operator with closure :
\begin{lem}\label{Compatibclosuretildeanalytic}
	Let $S$ be a semianalytic subset germ of $X_C$. Then,
$\overline{\widetilde{S}}=\widetilde{\overline{S}}$.				
\end{lem}
\begin{proof}
	One has obviously $\widetilde{S}\subset 
	\widetilde{\overline{S}}$. 
	By compacity of $C$, we get from 
	\cite[VIII Thm. 7.2]{ABR} that
	$\overline{S}$ is a compact subanalytic subset germ and it can be written as a \textit{finite} union of 
	sets of the form $\{f_{1}\geq 0,\ldots f_{r}\geq 0\}$, hence 
	$\widetilde{\overline{S}}$ is closed and 
	$\overline{\widetilde{S}}\subset\widetilde{\overline{S}}$.
	
	Let us show the converse inclusion.
	Let us consider $\alpha\in 	\widetilde{\overline{S}}$. 	
	Let $V$ be an open subset containing $\alpha$, one may assume that $V=\widetilde{U}$. Hence, 	$\alpha\in \widetilde{\overline{S}}\cap \widetilde{U}
	=\widetilde{\overline{S}\cap U}$.
	Since $\widetilde{\overline{S}\cap U}\not=\emptyset$, one has also 
	$\overline{S}\cap U\not=\emptyset$. Since $U$ is open, one gets ${S}\cap U\not=\emptyset$ and hence $\widetilde{S}\cap \widetilde{U}\not=\emptyset$. We have shown that $\widetilde{\overline{S}}\subset \overline{\widetilde{S}}$.
	
\end{proof}

We recall from \cite[VIII Thm. 8.4]{ABR}
that $A=\mathcal{O}(X_C)$ is an excellent ring.
From Proposition \ref{abstract764} and Lemma
\ref{Compatibclosuretildeanalytic}, one gets that
	$$\widetilde{\Cent(X_C)}=\Sp_c\mathcal{O}(X_C).$$

To conclude, we proceed as in the proof of Proposition \ref{Nashalypc-PSS}. Note again that the stability index coincides with the dimension by \cite[VIII Thm. 6.3]{ABR} in the analytic setting, and we get: 
\begin{prop}\label{ranalypc-PSS}(Analytic Central Positivstellens\"atze for surfaces):
	Let $f,f_1,\ldots,f_r$ in $A=\mathcal{O}(X_C)$ and assume $\dim X_C\leq 2$.	
	Let $\C=A\cap \sum \K(A)^2$. Then,
	\begin{enumerate}
		\item $f\geq 0$ on $\{f_1\geq 0,\ldots, f_r\geq 0\}\cap \Cent(X_C)$ if and only if $fq=p+f^{2m}$ where $p,q$ are in $\C[f_1,\ldots,f_r]$.
		\item $f>0$ on $\{f_1\geq 0,\ldots, f_r\geq 0\}\cap \Cent(X_C)$ if and only if $fq=1+p$ where $p,q$ are in $\C[f_1,\ldots,f_r]$.		
		\item $f=0$ on $\{f_1\geq 0,\ldots, f_r\geq 0\}\cap \Cent(X_C)$ if and only if $f^{2m}+p=0$
		where $p$ is in $\C[f_1,\ldots,f_r]$.		
	\end{enumerate}	
\end{prop}


\section{Sums of squares of rational continuous functions on the central spectrum}

From Artin solution of the 17th problem of Hilbert, we have already mentioned that if $p\in R[x_1,\ldots,x_n]$ is non-negative on $R^n$ then $p$ is a sum of squares of rational functions. From \cite{Kre}, it is possible to find such a sum of squares such that the rational functions appearing can be extended continuously to $R^n$ for the Euclidean topology i.e are rational continuous functions on $R^n$. We denote by $\K^0(R^n)$ the ring of rational continuous functions on $R^n$. Rational continuous functions are called regulous when they are still rational continuous by restriction to any subvariety. Rational continuous and regulous functions are introduced and studied in \cite{KN}, \cite{FHMM}, \cite{Mo} and \cite{BFMQ}.
The above result involving sums of squares of rational continuous functions can be generalized as it is done in \cite[Thm. 6.1]{FMQ}, namely $f\in\K^0(R^n)$ is nonnegative on $R^n$ if and only if $f\in\sum\K^0(R^n)^2$. The proof of this result in \cite{FMQ} is over $\RR$ but it is also valid over any real closed field in place of $\R$. 

\smallskip

We may wonder if it is possible to get a continuity property in Theorem \ref{precentralHilbert17} and Proposition \ref{centralHilbert17}. As continuity is a topological notion, we choose here to only look at a continuous version of Proposition \ref{centralHilbert17}, the one associated with central orderings of a topological nature. To do that, we recall some material about abstract continuity on the real spectrum which is mainly taken from \cite{ABR}.

Any $f\in A$ may be associated to a function defined on $\Sp_r A$, assigning $\alpha\mapsto f(\alpha)\in R_\alpha$ where $R_\alpha$ is a real closure of $k(\supp(\alpha))$. It does not give functions in the usual sense since the $R_\alpha$'s vary.
One may then define abstract semialgebraic functions on $\Sp_r A$ given by a first order formula with parameters in $A$ (see \cite[II 5]{ABR}). For instance, for $p\in A$ and $q\in A\setminus\{0\}$, one may define the abstract semialgebraic function $f$ by setting $f(\alpha)=\frac{p(\alpha)}{q(\alpha)}$ whenever $q(\alpha)\not=0$ and $f(\alpha)=0$ otherwise.
One may also define functions on a proconstructible subset $Y$ of $\Sp_r A$. 
To end, one says that an abstract semialgebraic function $f$ is continuous on $Y$ if, for any specialization $\beta\to \alpha$ in $Y$, one has $f(\beta)\in W_{\beta\alpha}$ and 
$\lambda_{\beta\alpha}(f(\beta))=f(\alpha)$ where $W_{\beta\alpha}$ and $\lambda_{\beta\alpha}$ are respectively the valuation ring and the place associated to the specialization $\beta\to \alpha$ (see \cite[II 3.10]{ABR}). Let $f$ be an abstract semialgebraic function on $\Sp_c A$, we say that $f$ is rational continuous on $\Sp_{c} A$
if $f$ is continuous on $\Sp_{c} A$ and if there exist $p\in A$ and $q\in A\setminus\{0\}$, such that $f(\alpha)=\frac{p(\alpha)}{q(\alpha)}$ whenever $\alpha\in\Sp_c A\setminus \Zi^c (q)$ . We denote by $\K^0(\Sp_{c} A)$ the ring of rational continuous functions on $\Sp_c A$.

\begin{prop}
\label{integre}
The ring $\K^0(\Sp_c A)$ is a domain whose fraction field is $\K(A).$
\end{prop}

\begin{proof}
Consider the map $\K^0 (\Sp_{c} A)\to \K(A)$ which send $f\in\K^0 (\Sp_{c} A)$ to the class of $p/q$ in $\K(A)$ with $p\in A$, $0\not=q\in A$, and $f(\alpha)=\frac{p(\alpha)}{q(\alpha)}$ whenever $\alpha\in\Sp_c A\setminus\Zi^c(q)$. We have to prove this map is injective and thus we assume the class of $p/q$ in $\K(A)$ is $0$. It follows that $pq$ vanishes on $\Sp_c A$ i.e $pq\in\I (\Zi^c ((0)))$. Since $\K(A)$ is formally real then $(0)$ is a central ideal of $A$ (Proposition \ref{existcentralideal}) and thus, using the central Nullstellensatz (Proposition \ref{CentralNSS}) we get $pq=0$ in $A$. Since $A$ is a domain and $q\not=0$ then it follows that $p=0$. Let $\alpha\in\Sp_c A$. If $\alpha\not\in \Zi^c(q)$ then $f(\alpha)=\frac{p(\alpha)}{q(\alpha)}=0$. Assume now $\alpha\in\Zi^c(q)$. By definition of a central ordering and by Proposition \ref{orderingnull}, there is $\beta\in\Sp_c A$ such that $\supp(\beta)=(0)$ and $\beta\to\alpha$. Clearly $\beta \not\in\Zi^c(q)$ and thus $f(\beta)=\frac{p(\beta)}{q(\beta)}=0\in W_{\beta\alpha}$ and 
$\lambda_{\beta\alpha}(f(\beta))=f(\alpha)=0$. It follows that $f=0$ in $\K^0(\Sp_c A)$ and the proof is done.

Since $A\subset \K^0 (\Sp_c A)\subset \K(A)$ then it follows that $\K(\K^0 (\Sp_c A))=\K(A)$.
\end{proof}

If $V$ an irreducible irreducible algebraic variety over $R$, we denote by $\K^0(\Cent V(R))$ the ring of rational continuous functions on $\Cent V(R)$. These functions are defined in the same way as for $R^n$ (see \cite{FMQ-futur2}). It is then clear that the restriction to $\Cent V(R)$ of an element of $\K^0(\Sp_c R[V])$ is in $\K^0(\Cent V(R))$.

\smallskip

Let us define the following subcone of $\C=A\cap \sum \K(A)^2$ which is convenient to deal with continuity: $$\C^0=A\cap \sum \K^0(A)^2.$$

Adding a continuity property in Theorem \ref{precentralHilbert17} and Proposition \ref{centralHilbert17} would suggest, for a given $f\in A$, that 
$f(\Sp_c A)\geq 0$ if and only if $f\in\C^0$. 
Unfortunately, this equivalence is false i.e, in general, an element of $A$ nonnegative on $\Sp_{c} A$ is not necessarily a sum of squares in  $\K^0(\Sp_{c}A)$ as it is shown by the following: 

\begin{ex}
	\label{Whitney}
	Let $V$ be the Whitney umbrella with coordinate ring $A=\RR[V]=\RR[x,y,z]/(y^2-zx^2)$. Its normalization $V'$ is smooth and we have $B=\RR[V']=\RR[x,Y,z]/(Y^2-z)$. The ring morphism $A\to B$ associated to the normalization map $\pi':V'\to V$ is given by $x\mapsto x$, $y\mapsto Yx$ and $z\mapsto z$.
	Let $f=z\in A$. Since $f=(x/y)^2\in\K(A)^2$ then it follows from Proposition \ref{centralHilbert17} that $f\geq 0$ on $\Sp_{c} A$. 	
	\smallskip
	
	We prove now that $f\not\in\C^0$. Assume $f\in\C^0$ then we get, by restriction to $\Cent(V(\RR))$, 
	$$f=\sum_{i=1}^n f_i^2$$ with $f_i\in\K^0(\Cent V(\R))$. By composition with $\pi'_{\mid \RR}:V'(\RR)\to V(\RR)$ then we get $$g=f\circ\pi'=\sum\limits_{i=1}^n g_i^2$$ with $g_i=f_i\circ \pi'_{\mid \RR}\in\K^0(V'(\R))^2$ (note that $\pi'_{\mid \RR}$ is surjective on $\Cent V(\RR)$). Since $V'$ is smooth then the $g_i$ are regulous functions and thus are still rational continuous
	by restriction to a subvariety \cite{KN}. So we restrict our identity $g= \sum_{i=1}^n g_i^2$ to the curve $C\subset V'(\R)$ with equations $x=0$ and $Y^2=z$ and since the curve is smooth then the restriction of the $g_i$ to $C$ are regular functions \cite{Mo}. On $C$ we get 
	$$g=z=\sum_{i=1}^n (\frac{a_{i,1}(z)+a_{i,2}(z)Y}{b_{i,1}(z)+b_{i,2}(z)Y})^2$$
	with the $a_{i,j}$ and $b_{i,j}$ polynomials in $z$. But the fractions $\frac{a_{i,1}(z)+a_{i,2}(z)Y}{b_{i,1}(z)+b_{i,2}(z)Y}$ are composition by $\pi'_{\mid \RR}$ of continuous functions on the superior half of the $z$-axis in $V(\RR)$ and thus it follows from an easy calculus that $a_{i,2}=b_{i,2}=0$ for $i=1,\ldots,n$. The previous identity becomes impossible and the proof is done.
\end{ex}

\smallskip

Let us see now that we get this continuity property if and only if some $p$ and $q$ appearing in the identities of the second and the third statements of Proposition \ref{centralHilbert17} belong to $\C^0$. 

\begin{prop} \label{centralHilbert17continu}
	Let $f\in A$. The following properties are equivalent:
	\begin{enumerate}
		\item There exist $p,q$ in $\C^0$ such that $fq=p+f^{2m}$.
		\item There exist $p,q\in\C^0$ such that $q^2f=p$ and $\Zi^{c}(q)\subset\Zi^{c}(f)$.
		\item $f\in\C^0$.
	\end{enumerate}
\end{prop}

\begin{proof}
		Let us show (1) implies (2).
	Suppose there exist $p,q$ in $\C^0$ such that $fq=p+f^{2m}$. One may assume that $f\not=0$ and then $p+f^{2m}\not=0$ since $\K(A)$ is formally real.
	We set $P=(f^2q)(p+f^{2m})$ and $Q=p+f^{2m}$.  Clearly, $P,Q\in\C^0$.
	Following the proof of Theorem \ref{precentralHilbert17} and Proposition  \ref{centralHilbert17} 
	we get $Q^2f=P$ and $\Zi^{c}(Q)\subset \Zi^{c}(f)$.
	
	We prove (2) implies (3).	
	Assume there exist $P,Q\in\C^0$ such that $Q^2f=P$ and $\Zi^{c}(Q)\subset\Zi^{c}(f)$. Hence, one may write $f=\sum f_i^2$ where $f_i=\frac{g_i}{Q}$ with $g_i\in A$. Clearly $f_i\in\K(A)$.
	Considering that $f_i(\alpha)=0$ whenever $Q(\alpha)=0$, and $f_i(\alpha)=\frac{g_i(\alpha)}{Q(\alpha)}$ whenever $Q(\alpha)\not=0$, then $f_i$ is now defined on $\Sp_{c} A$.
	
	\smallskip
	
	We show now that $f_i\in\K^0(\Sp_{c} A)$ and we are left to prove it is continuous on $\Sp_c A$. Let $\beta\to \alpha$ be a specialization in $\Sp_{c}(A)$.
	
	The case $Q(\beta)=0$ is trivial. Indeed, $Q(\alpha)=0$ since $\beta$ specializes to $\alpha$ and in that case we have set $f_i(\beta)=0$ and $f_i(\alpha)=0$. Then, obviously $f_i(\beta)\in W_{\beta\alpha}$ and $\lambda_{\beta\alpha}(f_i(\beta))=f_i(\alpha)$.
	
	Let us assume in the following that $Q(\beta)\not=0$. In that case, we have set $f_i(\beta)=\frac{g_i(\beta)}{Q(\beta)}$. Since $W_{\beta\alpha}$ is $\beta$-convex and $f(\beta)=\sum f_i^2(\beta)$ is in $W_{\beta\alpha}$, one gets that $f_i^2(\beta)\in W_{\beta\alpha}$. Using that a valuation ring is integrally closed, one has  $f_i(\beta)\in W_{\beta\alpha}$, the first desired condition. Let us show now the second condition : $\lambda_{\beta\alpha}(f_i(\beta))=f_i(\alpha)$.
	
	First case: we assume $Q(\alpha)\not=0$. Since $Qf_i\in A$, one has $Q(\beta)f_i(\beta)\in W_{\beta\alpha}$, and
	$\lambda_{\beta\alpha}(Q(\beta)f_i(\beta))=Q(\alpha)f_i(\alpha)$
	and hence 
	$$\lambda_{\beta\alpha}(Q(\beta))\lambda_{\beta\alpha}(f_i(\beta))=Q(\alpha)f_i(\alpha)$$
	Since, $\lambda_{\beta\alpha}(Q(\beta))=Q(\alpha)\not=0$, one gets the desired condition
	$\lambda_{\beta\alpha}(f_i(\beta))=f_i(\alpha)$.
	
	Second case: assume that $Q(\alpha)=0$. Since $\Zi^{c}(Q)\subset \Zi^{c}(f)$, hence $f(\alpha)=0$. 
	Since $f=\sum f_i^2$, one gets 
	$$0=f(\alpha)=\lambda_{\beta\alpha}(f(\beta))=
	\sum \lambda_{\beta\alpha}(f_i(\beta))^2	$$
	This shows that, for any $i$,
	$\lambda_{\beta\alpha}(f_i(\beta))=0$, which gives $\lambda_{\beta\alpha}(f_i(\beta))=f_i(\alpha)$ and we have proved that (2) implies (3).
	
	(3) implies (1).
	Assume $f\in\C^0$. We set $q=f\in\C^0$ and $p=0$ and we get the identity of (1) namely $fq=p+f^{2m}$ for $m=1$.
\end{proof}

One sufficient condition to fit in the hypothesis of this proposition is to assume non negativity of our element $f$ on the whole real spectrum (not only on the central spectrum). Namely, one gets the following version of a central continuous Hilbert 17th property :

\begin{thm}\label{continuousHilbert}
	Let $f\in A$. If $f\geq 0$ on 
	$\Sp_r A$, then $f\in\C^0$.	
\end{thm}

\begin{proof}
Assume $f\geq 0$ on $\Sp_r A$. By the formal positivstellensatz we get an identity $fq=p+f^{2m}$ with $p,q\in\sum A^2\subset \C^0$. We conclude using Proposition \ref{centralHilbert17continu}.
\end{proof}

Using the Artin-Lang property, one may derive 
a geometric version of Theorem \ref{continuousHilbert} (left to the reader) and of Proposition \ref{centralHilbert17continu}, namely: 
\begin{prop} \label{geomcentralHilbert17continu}
	Let $V$ be an irreducible affine algebraic variety over $R$ with $V_{reg}(R)\not=\emptyset$. Let $f\in R[V]$. The following properties are equivalent:
	\begin{enumerate}
		\item There exist $p,q$ in $R[V]\cap\sum \K^0(\Cent V(R))^2$ such that $fq=p+f^{2m}$.
		\item There exist $P,Q\in R[V]\cap\sum \K^0(\Cent V(R))^2$ such that $Q^2f=P$ and $\Zi(Q)\cap \Cent V(R)\subset\Zi(f)\cap \Cent V(R)$.
		\item $f\in R[V]\cap\sum \K^0(\Cent V(R))^2$.
	\end{enumerate}
\end{prop}

Let $V$ be the Cartan umbrella of coordinate ring $\R[V]=\R[x,y,z]/(x^3 -z(x^2 + y^2))$ and 
$f= x^2 + y^2 - z^2$. As already discussed in Example \ref{CartanEx}, there is $Q=x^2+y^2\in \C^0$ and $P=3x^4y^2+3x^2y^4+y^6\in \C^0$ such that $Q^2f=P$. Since, $\Zi(Q)\cap \Cent V(\R)\subset \Zi(f)\cap \Cent V(\R)$ one gets that $f\in\sum\K^0(\Cent V(\R))^2$.

Note that $f$ is not nonnegative on all $V(\RR)$ which says that the given version of central Hilbert 17th property as Theorem \ref{continuousHilbert} shall be refined.



\begin{thebibliography}{ABR2}     

\bibitem{ABR} C. Andradas, L. Br\"ocker, J.M. Ruiz, {\it Constructible sets in real geometry}, Springer (1996)




\bibitem{BP} E. Becker, V. Powers, {\it Sums of powers in rings and the real holomorphy ring}, J. reine angew. Math. 480 (1996), 71-103 

\bibitem{BFMQ} F. Bernard, G. Fichou, J.-P. Monnier, R. Quarez,
{\it Algebraic characterization of homeomorphisms between algebraic varieties}, Arxiv (2022)

\bibitem{BCR} J. Bochnak, M. Coste, M.-F. Roy, 
{\it Real algebraic
  geometry}, Springer, (1998)


\bibitem{Br} L. Bröcker, {\it On basic semialgebraic sets}, Expo. Math., 9, 289-334 (1991)

\bibitem{Du} D. W. Dubois,
{\it Real commutative algebra I. Places.}
Rev. Mat. Hisp.-Amer. (4) 39, no 2-3, (1979)


\bibitem{FHMM} G. Fichou, J. Huisman, F. Mangolte,
  J.-P. Monnier,
{\it Fonctions r\'egulues}, J. Reine angew. Math., 718, 103-151 (2016)

\bibitem{FMQ} G. Fichou, J.-P. Monnier, R. Quarez,
{\it Continuous functions on the plane regular after one blowing-up}, 
Math. Z., 285, 287-323, (2017)



\bibitem{FMQ-futur2} G. Fichou, J.-P. Monnier, R. Quarez, {\it
    Weak and semi normalization in real algebraic geometry},
 Ann. Sc. Norm. Super. Pisa Cl. Sci. (5) 22, 1511-1558, (2021)
 
 

 
 \bibitem{KN} J. Koll\'ar, K. Nowak,
{\it Continuous rational functions on real and p-adic varieties},
Math. Z. 279, 1-2, 85-97 (2015).

 
 \bibitem{Kre} G. Kreisel, Review of Ershov, Zbl. 374, 02027 (1978)
 
 \bibitem{L} T. Y. Lam, 
 {\it An introduction to real algebra},
 Rocky Montain J. Math. (4) 14, 767-814, (1984)




\bibitem{Ma} H. Matsumura,
{\it Commutative algebra}, Cambridge studies in advanced mathematics
8, (1989)

\bibitem{Mo} J.-P. Monnier,
{\it Semi-algebraic geometry with rational continuous functions},
Math. Ann. 372, 3-4, 1041-1080 (2018). 


\bibitem{Mnew} J.-P. Monnier, {\it Central algebraic geometry and seminormality}, To appear in Rend. Sem. Mat. Univ. Padova, (2022)


\bibitem{She} C. Scheiderer, {\it Stability index of real varieties}, Inventiones Math. 97, no. 3, 467-483, 1989

\bibitem{T} C. Traverso, {\it Seminormality and Picard group}, Ann. Scuola Norm. Sup. Pisa (3) 24, 585--595, (1970)



\end{thebibliography}
\end{document}